\documentclass[12pt,oneside,reqno]{amsart}

\usepackage{hyperref}

\usepackage[headsep=30pt,footskip=30pt,twoside=false,bottom=2.5cm]{geometry}

\usepackage[all]{xy}

\xymatrixcolsep{2cm}

\numberwithin{equation}{subsection}

\allowdisplaybreaks[1]

\usepackage{etoolbox}

\makeatletter

\patchcmd{\thesubsection}{\arabic}{\Alph}{}{}

\patchcmd{\@seccntformat}{\@secnumfont}{%
  \@secnumfont\expandafter\protect\csname format#1\endcsname}{}{}

\patchcmd{\@startsection}{\@afterindenttrue}{\@afterindentfalse}{}{}

\patchcmd{\subsection}{-.5em}{.3\linespacing}{}{}

\makeatother

\theoremstyle{plain}

\newtheorem{theorem}{Theorem}[section]

\newtheorem{proposition}[theorem]{Proposition}

\newtheorem{lemma}[theorem]{Lemma}

\newtheorem{corollary}[theorem]{Corollary}

\theoremstyle{remark}

\newtheorem{remark}[theorem]{Remark}

\newcommand{\sr}{\ensuremath{\mathrm{s}}}

\newcommand{\tg}{\ensuremath{\mathrm{t}}}

\newcommand{\Repr}[2][]{\ensuremath{\mathbf{Rep}_{#1}(#2)}}

\newcommand{\Hom}[3][]{\ensuremath{\mathrm{Hom}_{#1} (#2, #3)}}

\newcommand{\restrict}[2]{\ensuremath{#1\vert_{#2}}}

\newcommand{\Ker}[1]{\ensuremath{\mathrm{Ker} (#1)}}

\newcommand{\Coker}[1]{\ensuremath{\mathrm{Coker} (#1)}}

\newcommand{\Img}[1]{\ensuremath{\mathrm{Im} (#1)}}

\newcommand{\Coimg}[1]{\ensuremath{\mathrm{Coimg} (#1)}}

\newcommand{\cat}[1]{\ensuremath{\mathcal{#1}}}

\newcommand{\End}[2][]{\ensuremath{\mathrm{End}_{#1} (#2)}}

\newcommand{\Limp}{\ensuremath{\Rightarrow}}

\newcommand{\id}[1]{\ensuremath{\mathbf{1}_{#1}}}

\renewcommand{\dim}[2][]{\ensuremath{\mathrm{dim}_{#1}(#2)}}

\newcommand{\rk}[2][]{\ensuremath{\mathrm{rk}_{#1}(#2)}}

\newcommand{\N}{\ensuremath{\mathbf{N}}}

\newcommand{\Z}{\ensuremath{\mathbf{Z}}}

\newcommand{\Q}{\ensuremath{\mathbf{Q}}}

\newcommand{\R}{\ensuremath{\mathbf{R}}}

\newcommand{\Liff}{\ensuremath{\Leftrightarrow}}

\newcommand{\sym}[1]{\ensuremath{\mathrm{S}_{#1}}}

\newcommand{\set}[1]{\ensuremath{\{ #1 \}}}

\newcommand{\suchthat}{\ensuremath{\, \vert \,}}

\newcommand{\supp}[1]{\ensuremath{\mathrm{Supp} (#1)}}

\newcommand{\closure}[1]{\ensuremath{\overline{#1}}}

\newcommand{\sgn}[1]{\ensuremath{\mathrm{sgn}(#1)}}

\newcommand{\C}{\ensuremath{\mathbf{C}}}

\newcommand{\tr}[1]{\ensuremath{\mathrm{Tr}(#1)}}

\newcommand{\card}[1]{\ensuremath{\mathrm{card} (#1)}}

\newcommand{\struct}[1]{\ensuremath{\mathcal{O}_{#1}}}

\newcommand{\sstalk}[2]{\ensuremath{\struct{#1,#2}}}

\newcommand{\HOM}[3][]{%
  \ensuremath{\mathcal{H}\mathit{om}_{#1}(#2, #3)}}

\newcommand{\pullback}[2]{\ensuremath{#1^*(#2)}}

\newcommand{\pair}[2]{\ensuremath{\langle #1, #2 \rangle}}

\newcommand{\pr}[1]{\ensuremath{\mathrm{pr}_{#1}}}

\newcommand{\Lie}[1]{\ensuremath{\mathrm{Lie}(#1)}}

\newcommand{\htang}[2][]{\ensuremath{\mathrm{T}_{#1}(#2)}}

\newcommand{\norm}[1]{\ensuremath{\lVert#1\rVert}}

\newcommand{\Aut}[2][]{\ensuremath{\mathrm{Aut}_{#1} (#2)}}

\newcommand{\units}[1]{\ensuremath{#1^\times}}

\renewcommand{\bar}[1]{\ensuremath{\overline{#1}}}

\newcommand{\abs}[1]{\ensuremath{\lvert#1\rvert}}

\newcommand{\tang}[2][]{\ensuremath{\mathrm{T}_{#1}(#2)}}

\newcommand{\vect}[1]{\ensuremath{\mathrm{V}(#1)}}

\newcommand{\Ann}[1]{\ensuremath{\mathrm{Ann}(#1)}}

\DeclareMathOperator{\diff}{d}

\DeclareMathOperator{\Ad}{Ad}

\newcommand{\schur}[1]{\ensuremath{#1_{\mathrm{schur}}}}

\newcommand{\stab}[1]{\ensuremath{#1_{\mathrm{s}}}}

\newcommand{\eh}[1]{\ensuremath{#1_{\mathrm{eh}}}}

\newcommand{\irr}[1]{\ensuremath{#1_{\mathrm{irr}}}}

\newcommand{\zm}[1]{\ensuremath{#1_{\mathrm{m}}}}

\newcommand{\zms}[1]{\ensuremath{#1_{\mathrm{ms}}}}

\newcommand{\UU}[1]{\ensuremath{\mathrm{U}(#1)}}

\begin{document}

\title[Moduli of quiver representations]{Holomorphic aspects of moduli
  of representations of quivers}

\author{Pradeep Das}

\author{S.~Manikandan}

\author{N.~Raghavendra}

\email{pradeepdas@hri.res.in, smanikandan@hri.res.in,
  raghu@hri.res.in}

\address{Harish-Chandra Research Institute \\ HBNI \\ Chhatnag Road \\ Jhunsi \\
  Allahabad 211~019 \\ India}



\begin{abstract}
  This article describes some complex-analytic aspects of the moduli
  space of the finite-dimensional complex representations of a finite
  quiver, which are stable with respect to a fixed rational weight.
  We construct a natural structure of a complex manifold on this
  moduli space, and a K\"ahler metric on the complex manifold.  We
  then define a Hermitian holomorphic line bundle on the moduli space,
  and show that its curvature is a rational multiple of the K\"ahler
  form.
\end{abstract}

\maketitle

\section*{Introduction}
\label{sec:introduction}

Moduli spaces of representations of quivers are of interest because of
their relations with the moduli spaces of representations of algebras
\cite{KMR}, and with the moduli spaces of sheaves on projective
schemes \cite{ACK}.  A general survey about the moduli spaces of
representations of quivers is \cite{RMR}.

In this paper, we discuss some complex-analytic aspects of the moduli
space of the finite-dimensional complex representations of a finite
quiver, which are stable with respect to a fixed rational weight.  We
describe a natural K\"ahler metric on this moduli space, and exhibit a
Hermitian holomorphic line bundle on it, whose Chern form is
essentially an integral multiple of the K\"ahler form of this metric.
This integral multiple depends only on the chosen weight.

The methods of this paper are elementary in nature.  They are based on
K\"ahler geometry, and do not use any results from geometric invariant
theory.  In particular, when the moduli spaces of stable
representations are compact, by the Kodaira embedding theorem, the
results of this paper give an analytic proof of the projectivity of
these moduli spaces.

We view the stability of representations of quivers as a special case
of Rudakov's theory of stability structures on an abelian category.
Accordingly, we begin by recalling this theory in
Section~\ref{sec:stability-structures}.  Any finite positive family of
additive functions on an abelian category, and a corresponding family
of real numbers, called a \emph{weight}, define an stability structure
on the category.  There is a natural hyperplane arrangement on the
space of weights, and the stability condition remains constant within
every facet of this hyperplane arrangement.  We describe this idea in
Section~\ref{sec:stab-with-resp}.  It includes, as a special case, the
stability of representations of a finite quiver with respect to a
given weight.

We recall some basic notions about quivers and their representations
in Section~\ref{sec:repr-quiv}.  We also describe a theorem of King,
which relates stability of a representation of a quiver to the
existence of a certain kind of inner product on the representation,
which we call an \emph{Einstein-Hermitian metric}, because of its
similarity to Einstein-Hermitian metrics on vector bundles.  We
discuss families of representations in Section~\ref{sec:famil-repr},
and explain a criterion for two representations in a family to be
separated from each other.

We construct the moduli space of Schur representations in
Section~\ref{sec:moduli-space-schur}.  It is, in general, a
non-Hausdorff complex manifold.  Its open subset of stable
representations is Hausdorff, and has a natural K\"ahler metric, as we
explain in Section~\ref{sec:kahler-metric-moduli}.  We end the paper
with a description of a natural Hermitian holomorphic line bundle on
the moduli space of representations that are stable with respect to a
rational weight, and show that its curvature is essentially an
integral multiple of the K\"ahler form on the moduli space.

\section{Stability structures}
\label{sec:stability-structures}

The stability of representations of a quiver is a special case of the
notion of a stability structure that was defined by Rudakov
\cite[Definition 1.1]{RSA}.  In this section, we recall some
properties of such stability structures.  We first look at the
definition of stability structures.  Then, we recall the Schur Lemma
about the endomorphisms of stable objects.  Lastly, we mention
Jordan-H\"older and Harder Narasimhan filtrations, and the notion of
$S$-equivalence of semistable objects.

\subsection{Semistable objects of an abelian category}
\label{sec:semist-objects-an}

Let $\cat{A}$ be an abelian category, and $\preceq$ a total preorder
on the set of non-zero objects of $\cat{A}$.  For any two non-zero
objects $M$ and $N$ of $\cat{A}$, write $M \succeq N$ if
$N \preceq M$, and define
\begin{align*}
  M \prec N
  &\quad\text{if}\quad M \preceq N ~\text{and}~ M \not\succeq N, \\
  M \asymp N
  &\quad\text{if}\quad M \preceq N ~\text{and}~ M \succeq N, \\
  M \succ N
  &\quad\text{if}\quad M \succeq N ~\text{and}~ M \not\preceq N.
\end{align*}
Then, $\succeq$ is a total preorder, $\prec$ and $\succ$ irreflexive
transitive relations, and $\asymp$ an equivalence relation, on the set
of non-zero objects of $\cat{A}$.  Moreover, for any two non-zero
objects $M$ and $N$ of $\cat{A}$, exactly one of the three statements
\begin{equation*}
  \label{eq:1}
  M \prec N, \quad
  M \asymp N, \quad
  M \succ N
\end{equation*}
holds.  We say that $\preceq$ has the \emph{seesaw property} if for
every short exact sequence
\begin{equation*}
  \label{eq:2}
  0 \to M' \xrightarrow{f'} M \xrightarrow{f} M'' \to 0
\end{equation*}
of non-zero objects of $\cat{A}$, exactly one of the three statements
\begin{equation*}
  \label{eq:3}
  M' \prec M \prec M'', \quad
  M' \asymp M \asymp M'', \quad
  M' \succ M \succ M''
\end{equation*}
is true.  A \emph{stability structure} on $\cat{A}$ is a total
preorder on the set of non-zero objects of $\cat{A}$, which has the
seesaw property.

We fix a stability structure $\preceq$ on an abelian category
$\cat{A}$.  The seesaw property implies that if $M$ and $M'$ are two
isomorphic non-zero objects of $\cat{A}$, then $M\asymp M'$.  It
follows that if $M\cong M'$ and $N\cong N'$ are isomorphisms of
objects in $\cat{A}$, then $M\preceq N$ (respectively, $M\prec N$,
$M\asymp N$) if and only if $M'\preceq N'$ (respectively,
$M'\prec N'$, $M'\asymp N'$).  In particular, if $i: N \to M$ and
$i' : N' \to M$ are two equivalent non-zero subobjects of $M$, we have
$N \preceq M$ (respectively, $N \prec M$, $N \asymp M$) if and only if
$N' \preceq M$ (respectively, $N' \prec M$, $N' \asymp M$).

An object $M$ of $\cat{A}$ is called \emph{semistable} (respectively,
\emph{stable}) if
\begin{equation*}
  \label{eq:7}
  N \preceq M
  \quad
  (\text{respectively,}~ N \prec M)
\end{equation*}
for every non-zero proper subobject $N$ of $M$.  We say that an object
of $\cat{A}$ is \emph{polystable} if it is semistable, and is
isomorphic to the direct sum of a finite family of stable objects of
$\cat{A}$.  It is obvious that stable $\Limp$ polystable $\Limp$
semistable, and that all three properties are preserved by
isomorphisms in $\cat{A}$.  It is also easy to verify the following
statements about semistable objects.

\begin{proposition}
  \label{pro:1}
  Let $\preceq$ be a stability structure on an abelian category
  $\cat{A}$.  Let $S$ be an $\asymp$-equivalence class in the set of
  non-zero objects of $\cat{A}$, and let $\cat{A}(S)$ be the full
  subcategory of $\cat{A}$, whose objects are either zero objects of
  $\cat{A}$, or semistable objects of $\cat{A}$ which belong to $S$.
  Then:
  \begin{enumerate}
  \item \label{item:1} A non-zero object $M$ of $\cat{A}$ is
    semistable (respectively, stable) if and only if for every
    non-zero epimorphism $f : M \to N$ which is not an isomorphism, we
    have
    \begin{equation*}
      \label{eq:8}
      M \preceq N \quad
      (\text{respectively,}~ M \prec N).
    \end{equation*}
  \item \label{item:2} Let
    \begin{equation*}
      \label{eq:9}
      0 \to M' \xrightarrow{f'} M \xrightarrow{f} M'' \to 0
    \end{equation*}
    be a short exact sequence of non-zero objects of $\cat{A}$.
    Suppose that
    \begin{equation*}
      \label{eq:10} M' \asymp M \asymp M''.
    \end{equation*}
    Then, $M$ is semistable if and only if both $M'$ and $M''$ are
    semistable.
  \item \label{item:3} Let $M$ and $N$ be two non-zero objects of
    $\cat{A}$.  Then, the object $M \oplus N$ is semistable if and
    only if both $M$ and $N$ are semistable, and $M \asymp N$.  In
    that case,
    \begin{equation*}
      \label{eq:11}
      (M \oplus N) \asymp M \asymp N.
    \end{equation*}
  \item \label{item:4} Let $M$ and $N$ be two semistable objects of
    $\cat{A}$, and let $f: M \to N$ be a morphism.  Suppose that
    $M \asymp N$.  Then, each of the objects $\Ker{f}$, $\Img{f}$,
    $\Coimg{f}$, and $\Coker{f}$, is either zero, or is semistable and
    $\asymp$-related to $M$.
  \item \label{item:5} The category $\cat{A}(S)$ is an abelian
    subcategory of $\cat{A}$.
  \end{enumerate}
\end{proposition}

\subsection{The Schur Lemma}
\label{sec:schur-lemma}

Recall that an object $M$ of an additive category $\cat{A}$ is called
\emph{simple} if it is non-zero, and has no non-zero proper subobject.
We say that $M$ is a \emph{Schur} object, or a \emph{brick}, if the
ring $\End{M}$ is a division ring.  It is easy to see that every
simple object of $\cat{A}$ is Schur.  The following Proposition
follows directly from \cite[Theorem 1]{RSA}.

\begin{proposition}
  \label{pro:2}
  Let $\cat{A}$, $\preceq$, $S$, and $\cat{A}(S)$, be as in
  Proposition~\ref{pro:1}.
  \begin{enumerate}
  \item \label{item:18} Let $M$ and $N$ be two semistable objects of
    $\cat{A}$, and let $f: M \to N$ be a non-zero morphism.  Suppose
    that $M \succeq N$.  Then:
    \begin{enumerate}
    \item \label{item:19} $M \asymp N$.
    \item \label{item:20} If $M$ is stable, then $f$ is a
      monomorphism.
    \item \label{item:21} If $N$ is stable, then $f$ is an
      epimorphism.
    \item \label{item:22} If both $M$ and $N$ are stable, then $f$ is
      an isomorphism.
    \end{enumerate}
  \item \label{item:23} An element $M$ of $S$ is a simple object of
    the abelian category $\cat{A}(S)$ if and only if it is stable.
  \item \label{item:24} (Schur Lemma) Every stable object of $\cat{A}$
    is a Schur object of $\cat{A}$.
  \item \label{item:25} Suppose that $\cat{A}$ is a $K$-linear abelian
    category, where $K$ is an algebraically closed field, and let $M$
    be an object of $\cat{A}$.  Suppose also that the $K$-vector space
    $\End{M}$ is finite-dimensional.  Then, $M$ is a Schur object of
    $\cat{A}$ if and only if for every endomorphism $f$ of $M$ in
    $\cat{A}$, there exists a unique element $\lambda\in K$, such that
    $f = \lambda \id{M}$.
  \end{enumerate}
\end{proposition}

\subsection{Jordan-H\"older filtrations}
\label{sec:jord-hold-filtr}

A sequence $(M_n)_{n\in\N}$ of subobjects of an object $M$ of
$\cat{A}$ is called \emph{stationary} if there exists $n_0\in\N$, such
that $M_n = M_{n+1}$ for all $n\geq n_0$.  We say that an object $M$
of $\cat{A}$ is
\begin{enumerate}
\item \label{item:30} \emph{Noetherian} (respectively,
  \emph{Artinian}) if every sequence $(M_n)_{n\in\N}$ of subobjects of
  $M$, such that $M_n\subset M_{n+1}$ (respectively,
  $M_n\supset M_{n+1}$) for all $n\in\N$, is stationary.
\item \label{item:31} \emph{quasi-Noetherian} with respect to
  $\preceq$ if every sequence $(M_n)_{n\in\N}$ of subobjects of $M$,
  such that $M_n\subset M_{n+1}$, and $M_n\preceq M_{n+1}$ for all
  $n\in\N$, is stationary.
\item \label{item:32} \emph{weakly Artinian} with respect to $\preceq$
  if every sequence $(M_n)_{n\in\N}$ of subobjects of $M$, such that
  $M_n\supset M_{n+1}$, and $M_n\preceq M_{n+1}$ for all $n\in\N$, is
  stationary.
\item \label{item:33} \emph{weakly Noetherian} with respect to
  $\preceq$ if it is quasi-Noetherian with respect to $\preceq$, and
  if every sequence $(M_n)_{n\in\N}$ of subobjects of $M$, such that
  $M_n\subset M_{n+1}$, and $M_n\succeq M_{n+1}$ for all $n\in\N$, is
  stationary.
\end{enumerate}
The category $\cat{A}$ is called \emph{Noetherian} (respectively,
\emph{Artinian}) if every object in it is Noetherian (respectively,
Artinian).  It is called \emph{quasi-Noetherian} (respectively,
\emph{weakly Artinian}, \emph{weakly Noetherian}) with respect to
$\preceq$ if every object in it is quasi-Noetherian (respectively,
weakly Artinian, weakly Noetherian) with respect to $\preceq$.

If $M$ is a semistable object of $\cat{A}$, then a
\emph{Jordan-H\"older filtration} of $M$ with respect to $\preceq$ is
a sequence $(M_i)_{i=0}^n$ of subobjects of $M$, such that $n\geq 1$,
$M_0 = M$, $M_n=0$, and $M_i \subset M_{i-1}$, $M_{i-1}/M_i$ is
stable, and $M_{i-1} \asymp M$, for every $i=1,\dotsc,n$.  If $M$ is
an arbitrary object of $\cat{A}$, then a \emph{Harder-Narasimhan
  filtration} of $M$ with respect to $\preceq$ is a sequence
$(M_i)_{i=0}^n$ of subobjects of $M$, such that $n\in\N$, $M_0 = M$,
$M_n=0$, $M_i \subset M_{i-1}$ and $G_i = M_{i-1}/M_i$ is semistable
for every $i=1,\dotsc,n$, and $G_{i-1} \prec G_i$ for every
$i=2,\dotsc,n$.

The statements in the following Proposition are proved in
\cite[Theorems 2 and 3]{RSA}.

\begin{proposition}
  \label{pro:3}
  Let $\preceq$ be a stability structure on an abelian category
  $\cat{A}$.
  \begin{enumerate}
  \item \label{item:34} Suppose that $\cat{A}$ is quasi-Noetherian,
    and weakly Artinian, with respect to $\preceq$.  Then, every
    semistable object $M$ of $\cat{A}$ has a Jordan-H\"older
    filtration with respect to $\preceq$.  Moreover, if
    $(M_i)_{i=0}^n$ and $(N_j)_{j=0}^m$ are two Jordan-H\"older
    filtrations of $M$ with respect to $\preceq$, then $n=m$, and
    there exists a permutation $\pi\in \sym{n}$ such that
    $M_{i-1}/M_i$ is isomorphic to $N_{\pi(i)-1}/N_{\pi(i)}$ for every
    $i=1,\dotsc,n$.
  \item \label{item:35} Suppose that $\cat{A}$ is weakly Noetherian,
    and weakly Artinian, with respect to $\preceq$.  Then, every
    object of $\cat{A}$ has a unique Harder-Narasimhan filtration with
    respect to $\preceq$.
  \end{enumerate}
\end{proposition}

Let $M$ and $N$ be two semistable objects of $\cat{A}$.  Let
$(M_i)_{i=0}^n$ and $(N_j)_{j=0}^m$ be Jordan-H\"older filtrations of
$M$ and $N$, respectively, with respect to $\preceq$, which exist by
Proposition~\ref{pro:3}(\ref{item:34}).  Then, we say that $M$ is
\emph{$S$-equivalent} to $N$ with respect to $\preceq$, if $n=m$, and
there exists a permutation $\pi\in \sym{n}$, such that $M_{i-1}/M_i$
is isomorphic $N_{\pi(i)-1}/N_{\pi(i)}$ for every $i=1,\dotsc,n$.  By
the above Proposition, this is independent of the choices of the
Jordan-H\"older filtrations, and defines an equivalence relation on
the set of all semistable objects of $\cat{A}$.

\section{Stability with respect to a weight}
\label{sec:stab-with-resp}

We now consider a special kind of the stability structures defined in
Section \ref{sec:stability-structures}, namely stability structures
defined by a finite family of positive additive functions on an
abelian category, and a corresponding family of real numbers called
weights, which form a finite-dimensional real vector space, the weight
space.  Fixing the values of the additive functions defines a
hyperplane arrangement on the weight space.  We describe how
semistability and other related notions behave with respect to this
hyperplane arrangement.

\subsection{Semistablity with respect to a weight}
\label{sec:semist-with-resp}

Let $\cat{A}$ be an abelian category.  We say that a family
$(\phi_i)_{i\in I}$ of additive functions from the set of objects of
$\cat{A}$ to an ordered abelian group $G$ is \emph{positive} if
$\phi_i(M)\geq 0$ for every object $M$ of $\cat{A}$ and for every
$i\in I$, and if for every non-zero object $M$ of $\cat{A}$, there
exists an $i\in I$, such that $\phi_i(M)>0$.  In particular, we say
that an additive function $\phi$ from $\cat{A}$ to $G$ is
\emph{positive} if the singleton family $(\phi)$ is positive.  If
$\phi$ is a positive additive function from $\cat{A}$ to $G$, then an
object $M$ of $\cat{A}$ is zero if and only if $\phi(M)=0$, hence, if
$M'$ is a subobject of an object $M$, then $\phi(M') \leq \phi(M)$,
and equality holds if and only if $M'=M$.  The category $\cat{A}$ is
Noetherian and Artinian if there exists a positive additive function
from $\cat{A}$ to $\Z$.

We now fix a non-empty finite positive family $(\phi_i)_{i\in I}$ of
additive functions from $\cat{A}$ to $\Z$.  The \emph{dimension
  vector} of any object $M$ of $\cat{A}$ is the element $\phi(M)$ of
$\N^I$ that is defined by
\begin{equation*}
  \label{eq:24}
  \phi(M) = (\phi_i(M))_{i\in I}.
\end{equation*}
The \emph{rank} of $M$ is the natural number $\rk{M}$ defined by
\begin{equation*}
  \label{eq:25}
  \rk{M} = \sum_{i\in I}\phi_i(M).
\end{equation*}
Since $(\phi_i)_{i\in I}$ is a positive family of additive functions
from $\cat{A}$ to $\Z$, the function $\mathrm{rk}$ is a positive
additive function from $\cat{A}$ to $\Z$.

An element of the $\R$-vector space $\R^I$ is called a \emph{weight}
of $\cat{A}$.  We say that a weight is \emph{rational} (respectively,
\emph{integral}) if it belongs to the subset $\Q^I$ (respectively,
$\Z^I$) of $\R^I$.  We fix a weight $\theta$ of $\cat{A}$.  We define
the \emph{$\theta$-degree} of any object $M$ of $\cat{A}$ to be the
real number $\deg_\theta(M)$ given by
\begin{equation*}
  \label{eq:26}
  \deg_\theta(M) = \sum_{i\in I}\theta_i \phi_i(M).
\end{equation*}
If $M\neq 0$, we define another real number $\mu_\theta(M)$ by
\begin{equation*}
  \label{eq:27}
  \mu_\theta(M) = \frac{\deg_\theta(M)}{\rk{M}},
\end{equation*}
and call it the \emph{$\theta$-slope} of $M$.

\begin{proposition}
  \label{pro:4}
  Let $\theta$ be any weight of $\cat{A}$.  Define a relation
  $\preceq_\theta$ on the set of non-zero objects of $\cat{A}$, by
  setting $M \preceq_\theta N$ if $\mu_\theta(M) \leq \mu_\theta(N)$.
  Then, $\preceq_\theta$ is a stability structure on $\cat{A}$.
\end{proposition}

\begin{proof}
  Let $c(M) = \deg_\theta(M)$, and $r(M) = \rk{M}$, for every object
  $M$ of $\cat{A}$.  Then, $c$ is an additive function from $\cat{A}$
  to the ordered abelian group $\R$, and $r$ is a positive additive
  function from $\cat{A}$ to $\Z$.  Moreover, in the notation of
  \cite[Definition 3.1]{RSA}, the function $\mu_\theta$ is the
  $(c : r)$ slope, and the relation $\preceq_\theta$ the $(c : r)$
  preorder, on the set of non-zero objects of $\cat{A}$.  Therefore,
  it follows from \cite[Lemma 3.2 and Remark]{RSA} that
  $\preceq_\theta$ is a stability structure on $\cat{A}$.
\end{proof}

Let $\theta$ be a weight of $\cat{A}$.  An object of $\cat{A}$ is
called \emph{$\theta$-semistable} (respectively,
\emph{$\theta$-stable, $\theta$-polystable}) if it is semistable
(respectively, stable, polystable) with respect to the stability
structure $\preceq_\theta$ on $\cat{A}$.  If $\zeta$ is a strictly
positive real number, and $\omega=\zeta\theta$, then an object of
$\cat{A}$ is $\theta$-semistable (respectively, $\theta$-stable,
$\theta$-polystable) if and only if it is $\omega$-semistable
(respectively, $\omega$-stable, $\omega$-polystable).

There are obvious special versions of the statements in
Propositions~\ref{pro:1}--\ref{pro:2}, with semistable (respectively,
stable) objects replaced by $\theta$-semistable (respectively,
$\theta$-stable) objects of $\cat{A}$.  Moreover, since $\mathrm{rk}$
is a positive additive function from $\cat{A}$ to $\Z$, the category
$\cat{A}$ is Noetherian and Artinian.  In particular, by
Proposition~\ref{pro:3}, every $\theta$-semistable object of $\cat{A}$
has a Jordan-H\"older filtration, and every object of $\cat{A}$ has a
unique Harder-Narasimhan filtration, with respect to $\preceq_\theta$.
We say that two $\theta$-semistable objects of $\cat{A}$ are
\emph{$S_\theta$-equivalent} if they are $S$-equivalent with respect
to $\preceq_\theta$.

There is a well-known definition of the stability of objects of
$\cat{A}$ that is defined by King in \cite[p.~516]{KMR}.  Let
$\lambda$ be an additive function from $\cat{A}$ to $\R$.  Then, King
defines an object $M$ of $\cat{A}$ to be $\lambda$-semistable(respectively,
$\theta$-stable) if it is
non-zero, if $\lambda(M)=0$, and if $\lambda(N) \geq 0$ (respectively,
$\lambda(N) > 0$) for every non-zero proper subobject $N$ of $M$.  The
following Proposition shows that the notion of $\theta$-semistability
(respectively, $\theta$-stability) defined above is a special case of
this definition of $\lambda$-semistability (respectively,
$\lambda$-stability).

\begin{proposition}
  \label{pro:5}
  Let $\theta$ be a weight of $\cat{A}$, $\mu$ a real number, and $c$
  a strictly positive real number.  Define an additive function
  $\lambda$ from $\cat{A}$ to $\R$, by putting
  $\lambda(M) = c(\mu\,\rk{M} - \deg_\theta(M))$ for every object $M$
  of $\cat{A}$.  Let $O^{\mathrm{ss}}(\theta,\mu)$ (respectively,
  $O^{\mathrm{s}}(\theta,\mu)$) be the set of all $\theta$-semistable
  (respectively, $\theta$-stable) objects $M$ of $\cat{A}$, such that
  $\mu_\theta(M) = \mu$.  Let $K^{\mathrm{ss}}(\lambda)$
  (respectively, $K^{\mathrm{s}}(\lambda)$) be the set of all
  $\lambda$-semistable (respectively, $\lambda$-stable) objects of
  $\cat{A}$, in the sense of King.  Then,
  $O^{\mathrm{ss}}(\theta,\mu) = K^{\mathrm{ss}}(\lambda)$, and
  $O^{\mathrm{s}}(\theta,\mu) = K^{\mathrm{s}}(\lambda)$.
\end{proposition}

\begin{proof}
  The Proposition follows from the fact that $\rk{N}>0$ for every
  non-zero object $N$ of $\mathcal{A}$.
\end{proof}

\subsection{Facets with respect to a hyperplane arrangement}
\label{sec:facets-with-respect}

Let $E$ be an affine space modelled after a finite-dimensional
$\R$-vector space $T$.  We will give $T$ its usual topology.  A
\emph{hyperplane arrangement} in $E$ is a locally finite set of
hyperplanes in $E$.  We fix a hyperplane arrangement $\mathcal{H}$ in
$E$.  For any subset $X$ of $E$, we define
\begin{equation*}
  \label{eq:32}
  \mathcal{H}(X) =
  \set{H\in \mathcal{H} \suchthat H \cap X \neq \emptyset}.
\end{equation*}
If $X$ is a singleton $\set{x}$, we write $\mathcal{H}(x)$ for
$\mathcal{H}(X)$.

For every hyperplane $H$ in $E$, we define an equivalence relation
$\sim_H$ on $E$ by setting $x\sim_H y$ if $x$ and $y$ belong to $H$,
or if $x$ and $y$ are strictly on the same side of $H$.  We define
$\sim$ to be the equivalence relation on $E$, which is the
intersection of the relations $\sim_H$ as $H$ runs over $\mathcal{H}$.
The $\sim$-equivalence class of an element $a$ of $E$ is called the
\emph{facet} of $E$ through $a$ with respect to $\mathcal{H}$.

Let $F$ be a facet of $E$.  Then, for every point $a$ in $F$, and for
any element $H$ of $\mathcal{H}(a)$, we have $F\subset H$, hence
\begin{equation*}
  \label{eq:33}
  \mathcal{H}(F) = \mathcal{H}(a) =
  \set{H \in \mathcal{H} \suchthat F \subset H}.
\end{equation*}
In particular, since $\mathcal{H}$ is locally finite, the set
$\mathcal{H}(F)$ is finite.  The intersection $\supp{F}$ of all the
elements of $\mathcal{H}(F)$ is called the \emph{support} of $F$.  It
is an affine subspace of $E$, whose dimension is called the
\emph{dimension} of $F$, and is denoted by $\dim{F}$.  The closure
$\closure{F}$ of $F$ in $E$ is a subset of $\supp{F}$.

\begin{remark}
  \label{rem:4}
  Let $F$ be a facet of $E$, and $L$ its support.  Then, $F$ equals
  the interior of $\closure{F}$ in $L$.  In particular, $F$ is open in
  $L$ \cite[Chapter V, \S~1, no.~2, Proposition 3]{BLA2}.
\end{remark}

\subsection{The hyperplane arrangement on the weight space}
\label{sec:hyperpl-arrang-weigh}

Let $\cat{A}$ be an abelian category, $(\phi_i)_{i\in I}$ a non-empty
finite positive family of additive functions from $\cat{A}$ to $\Z$,
and $\mathrm{rk}$ a positive additive function from $\cat{A}$ to $\Z$,
as in Section \ref{sec:semist-with-resp}.  The $\R$-vector space
$\R^I$ is called the \emph{weight space} of $\cat{A}$.  We give it the
usual, that is, the product topology.

For all elements $\theta = (\theta_i)_{i\in I}$ of $\R^I$ and
$d = (d_i)_{i\in I}$ of $\N^I$, we define a real number
$\deg_\theta(d)$, and a natural number $\rk{d}$, by
\begin{equation*}
  \label{eq:34}
  \deg_\theta(d) = \sum_{i\in I} \theta_i d_i, \quad
  \rk{d} = \sum_{i\in I} d_i.
\end{equation*}
If $d$ is a non-zero element of $\N^I$, then $\rk{d}> 0$, so for each
$\theta\in \R^I$, we have a real number $\mu_\theta(d)$, which is
defined by
\begin{equation*}
  \label{eq:35}
  \mu_\theta(d) = \frac{\deg_\theta(d)}{\rk{d}}.
\end{equation*}
For any two non-zero elements $d$ and $e$ of $\N^I$, we define an
$\R$-linear function $f(d,e) : \R^I \to \R$ by
\begin{equation*}
  \label{eq:36}
  f(d,e)(\theta) = \mu_\theta(d) - \mu_\theta(e).
\end{equation*}
It is obvious that $f(d,e)=0$ if and only if $e\in\Q d$.

Fix a non-zero element $d$ of $\N^I$.  Let $S_d$ denote the set of all
elements $e$ of $\N^I\setminus \Q d$, for which there exist an object
$M$ of $\cat{A}$, and a subobject $N$ of $M$, such that $\phi(M) = d$
and $\phi(N) = e$.  Since the family $(\phi_i)_{i\in I}$ of additive
functions on $\cat{A}$ is positive, the set $S_d$ is contained in
$\prod_{i\in I}(\N\cap [0,d_i])$, and is hence finite.  For each
$e\in S_d$, let
\begin{equation*}
  \label{eq:41}
  H(d,e) =
  \Ker{f(d,e)} =
  \set{\theta\in \R^I \suchthat \mu_\theta(e) = \mu_\theta(d)}.
\end{equation*}
Since $e\notin \Q d$, the function $f(d,e)$ is non-zero, hence
$H(d,e)$ is a hyperplane in $\R^I$.  We thus get a finite hyperplane
arrangement
\begin{equation*}
  \label{eq:42}
  \mathcal{H}(d) =
  \set{H(d,e) \suchthat e\in S_d}
\end{equation*}
in the affine space $\R^I$.

We define $\mathrm{sgn}$ to be the function from $\R$ to the subset
$\set{-1,0,1}$ of $\R$, which is $-1$ at every negative real number,
vanishes at $0$, and is $1$ at every positive real number.

\begin{proposition}
  \label{pro:7}
  Let $F$ be a facet of $\R^I$ with respect to the hyperplane
  arrangement $\mathcal{H}(d)$, and let $\theta$ and $\omega$ be two
  elements of $F$.  Let $M$ be an object of $\cat{A}$, such that
  $\phi(M)=d$, and let $N$ be a non-zero subobject of $M$.  Then,
  \begin{equation*}
    \label{eq:45}
    \sgn{\mu_\theta(M)-\mu_\theta(N)} =
    \sgn{\mu_\omega(M)-\mu_\omega(N)}.
  \end{equation*}
\end{proposition}

\begin{proof}
  Let $e=\phi(N)$, and let $f = f(d,e) : \R^I \to \R$.  Then, for each
  weight $\lambda$ of $\cat{A}$, we have
  \begin{equation*}
    \label{eq:46}
    \mu_\lambda(M) - \mu_\lambda(N) =
    \mu_\lambda(d) - \mu_\lambda(e) =
    f(\lambda).
  \end{equation*}
  Therefore, we have to prove that
  $\sgn{f(\theta)} = \sgn{f(\omega)}$.  If
  $f(\theta) = f(\omega) = 0$, then the equality to be proved is
  obvious.  Suppose that either $f(\theta)$ or $f(\omega)$ is
  non-zero.  By interchanging $\theta$ and $\omega$, we can assume
  that $f(\theta) \neq 0$.  Then, $e\notin \Q d$.  As $\phi(M)=d$ and
  $\phi(N)=e$, this implies that $e\in S_d$.  Thus, the hyperplane
  $H=H(d,e)$ is an element of $\mathcal{H}(d)$.  Now, as $H=\Ker{f}$,
  $f$ is a defining function of $H$.  Moreover, $\theta \notin H$,
  since $f(\theta)\neq 0$.  As $\theta$ and $\omega$ both belong to
  the same facet $F$, we have $\theta \sim_H \omega$, so $\theta$ and
  $\omega$ are strictly on the same side of $H$.  Therefore,
  $f(\theta) f(\omega) > 0$.  It follows that
  $\sgn{f(\theta)} = \sgn{f(\omega)}$.
\end{proof}

\begin{proposition}
  \label{pro:8}
  Let $d$ be a non-zero element of $N^I$, $F$ a facet of $\R^I$ with
  respect to the hyperplane arrangement $\mathcal{H}(d)$, and $\theta$
  and $\omega$ two elements of $F$.  Let $M$ be an object of
  $\cat{A}$, such that $\phi(M) = d$.  Then, $M$ is
  $\theta$-semistable (respectively, $\theta$-stable) if and only if
  it is $\omega$-semistable (respectively, $\omega$-stable).
\end{proposition}

\begin{proof}
  As $\phi(M) = d$ is non-zero, and the family $(\phi_i)_{i\in I}$ is
  positive, $M$ is non-zero.  For each weight $\lambda$ of $\cat{A}$,
  and for each non-zero subobject $N$ of $M$, define
  \begin{equation*}
    \label{eq:47}
    g_\lambda(N) = \mu_\lambda(M) - \mu_\lambda(N).
  \end{equation*}
  Then, $M$ is $\lambda$-semistable if and only if
  $g_\lambda(N)\geq 0$, or equivalently, $\sgn{g_\lambda(N)}$ belongs
  to $\set{0,1}$, for every non-zero subobject $N$ of $M$.  Similarly,
  $M$ is $\lambda$-stable if and only if $\sgn{g_\lambda(N)} = 1$ for
  every non-zero proper subobject $N$ of $M$.  Therefore, it suffices
  to check that $\sgn{g_\theta(N)} = \sgn{g_\omega(N)}$ for every
  non-zero proper subobject $N$ of $M$.  As $\theta$ and $\omega$
  belong to the same facet $F$, this is a consequence of
  Proposition~\ref{pro:7}.
\end{proof}

\begin{proposition}
  \label{pro:9}
  Let $d$ be a non-zero element of $N^I$, $F$ a facet of $\R^I$ with
  respect to the hyperplane arrangement $\mathcal{H}(d)$, and $\theta$
  and $\omega$ two elements of $F$.  Let $M$ be an object of
  $\cat{A}$, such that $\phi(M)=d$.  Suppose that $M$ is
  $\theta$-semistable, and let $(M_i)_{i=0}^n$ be a Jordan-H\"older
  filtration of $M$ with respect to $\theta$.  Then, $M$ is
  $\omega$-semistable, and $(M_i)_{i=0}^n$ is a Jordan-H\"older
  filtration of $M$ with respect to $\omega$ also.
\end{proposition}

\begin{proof}
  The fact that $M$ is $\omega$-semistable has already been proved in
  Proposition~\ref{pro:8}. For every $i=1,\dotsc,n$, we have
  $\mu_\theta(M_{i-1})=\mu_\theta(M)$, hence, by
  Proposition~\ref{pro:7},
  \begin{equation*}
    \label{eq:48}
    \sgn{\mu_\omega(M)-\mu_\omega(M_{i-1})} =
    \sgn{\mu_\theta(M)-\mu_\theta(M_{i-1})} = 0,
  \end{equation*}
  so $\mu_\omega(M_{i-1}) = \mu_\omega(M)$.  It remains to prove that
  the quotient object $N_i = M_{i-1}/M_i$ is $\omega$-stable for every
  $i=1,\dotsc,n$.

  We will first verify that $\mu_\omega(N_i) = \mu_\omega(M)$ for all
  $i=1,\dotsc,n$.  If $i=n$, this follows from the above paragraph,
  since $N_n=M_{n-1}$.  Suppose $1\leq i\leq n-1$.  Then, both $i$ and
  $i+1$ belong to $\set{1,\dotsc,n}$, hence, by the above paragraph,
  \begin{equation*}
    \label{eq:49}
    \mu_\omega(M_i) = \mu_\omega(M_{i-1}) =
    \mu_\omega(M).
  \end{equation*}
  We also have a short exact sequence
  \begin{equation*}
    \label{eq:50}
    0 \to M_i \to M_{i-1} \to N_i \to 0,
  \end{equation*}
  of non-zero objects of $\cat{A}$.  Therefore, by the seesaw property
  of $\preceq_\omega$, we get
  \begin{equation*}
    \label{eq:51}
    \mu_\omega(N_i) = \mu_\omega(M_{i-1}) =
    \mu_\omega(M).
  \end{equation*}
  This proves that $\mu_\omega(N_i) = \mu_\omega(M)$ for all
  $i=1,\dotsc,n$.

  Let $i\in \set{1,\dotsc,n}$.  In view of the previous paragraph, to
  show that $N_i$ is $\omega$-stable, it suffices to show that
  $\mu_\omega(X) < \mu_\omega(M)$ for every proper non-zero subobject
  $X$ of $N_i$.  To begin with, since $N_i$ is $\theta$-stable, we
  have
  \begin{equation*}
    \label{eq:52}
    \mu_\theta(X) < \mu_\theta(N_i) = \mu_\theta(M).
  \end{equation*}
  Suppose first that $i=n$.  Then, $N_i=M_{n-1}$, and $X$ is a
  non-zero subobject of $M$, hence, by Proposition~\ref{pro:7} and the
  above inequality,
  \begin{equation*}
    \label{eq:53}
    \sgn{\mu_\omega(M) - \mu_\omega(X)} =
    \sgn{\mu_\theta(M) - \mu_\theta(X)} = 1.
  \end{equation*}
  It follows that $\mu_\omega(X) < \mu_\omega(M)$.  Suppose next that
  $1\leq i\leq n-1$.  Let $\pi : M_{i-1} \to N_i$ be the canonical
  projection, and let $Y = \pi^{-1}(X)$, that is, the kernel of the
  composite
  \begin{equation*}
    \label{eq:54}
    M_{i-1} \xrightarrow{\pi} N_i \to N_i/X.
  \end{equation*}
  Thus, $Y$ is a non-zero subobject of $M_{i-1}$, and we have a short
  exact sequence
  \begin{equation*}
    \label{eq:55}
    0 \to M_i \to Y \to X \to 0
  \end{equation*}
  of non-zero objects of $\cat{A}$.  By the above paragraphs,
  \begin{equation*}
    \label{eq:56}
    \mu_\theta(M_i) = \mu_\theta(M) =
    \mu_\theta(N_i) > \mu_\theta(X).
  \end{equation*}
  Therefore, by the seesaw property of $\preceq_\theta$,
  \begin{equation*}
    \label{eq:57}
    \mu_\theta(M) = \mu_\theta(M_i) > \mu_\theta(Y),
  \end{equation*}
  hence, by Proposition~\ref{pro:7},
  \begin{equation*}
    \label{eq:58}
    \sgn{\mu_\omega(M) - \mu_\omega(Y)} =
    \sgn{\mu_\theta(M) - \mu_\theta(Y)} = 1,
  \end{equation*}
  so
  \begin{equation*}
    \label{eq:59}
    \mu_\omega(M_i) = \mu_\omega(M) > \mu_\omega(Y).
  \end{equation*}
  Again, by the seesaw property of $\preceq_\omega$, we get
  $\mu_\omega(M_i) > \mu_\omega(X)$, hence
  $\mu_\omega(M) > \mu_\omega(X)$.  This proves that $N_i$ is
  $\omega$-stable.
\end{proof}

\begin{proposition}
  \label{pro:10}
  Let $d$ be a non-zero element of $N^I$, $F$ a facet of $\R^I$ with
  respect to the hyperplane arrangement $\mathcal{H}(d)$, and $\theta$
  and $\omega$ two elements of $F$.  Let $M$ be an object of
  $\cat{A}$, such that $\phi(M) = d$.  Then, $M$ is
  $\theta$-polystable if and only if it is $\omega$-polystable.
\end{proposition}

\begin{proof}
  Suppose $M$ is $\theta$-polystable.  Then, $M$ is
  $\theta$-semistable, and there exists a sequence $(M_i)_{i=1}^n$ of
  $\theta$-stable objects of $\cat{A}$, such that $n\in\N$, $M_i$ is
  $\theta$-stable for each $i=1,\dotsc,n$, and $M$ is isomorphic to
  $\bigoplus_{i=1}^n M_i$.  As $M$ is non-zero, we in fact have
  $n\geq 1$.  Also, by Proposition~\ref{pro:8}, $M$ is
  $\omega$-semistable.  Let $N = \bigoplus_{i=1}^n M_i$.  Then, since
  $N$ is isomorphic to $M$, it is both $\theta$-semistable and
  $\omega$-semistable, and $\phi(N) = d$.  For each $i=0,\dotsc,n$,
  define $N_i = \bigoplus_{j=i+1}^n M_i$.  Then, $(N_i)_{i=0}^n$ is a
  decreasing sequence of subobjects of $N$, $N_0 = N$, $N_n=0$, and
  for each $i=1,\dotsc,n$, $N_{i-1}/N_i$ is isomorphic to $M_i$, and
  is hence $\theta$-stable.  Moreover, since $N$ is
  $\theta$-semistable, by Proposition~\ref{pro:1}(\ref{item:3}), for
  every $i=1,\dotsc,n$, we get
  \begin{equation*}
    \label{eq:60}
    \mu_\theta(N) = \mu_\theta(M_i) =
    \mu_\theta(N_{i-1}).
  \end{equation*}
  Therefore, $(N_i)_{i=0}^n$ is a Jordan-H\"older filtration of $N$
  with respect to $\theta$.  By Proposition~\ref{pro:9}, it is a
  Jordan-H\"older filtration of $N$ with respect to $\omega$ also.  In
  particular, for each $i=1,\dotsc,n$, $N_{i-1}/N_i$ is
  $\omega$-stable, hence $M_i$ is $\omega$-stable.  As $M$ is
  $\omega$-semistable, and isomorphic to $\bigoplus_{i=1}^nM_i$, it
  follows that $M$ is $\omega$-polystable.
\end{proof}

\begin{proposition}
  \label{pro:11}
  Let $d$ be a non-zero element of $N^I$, $F$ a facet of $\R^I$ with
  respect to the hyperplane arrangement $\mathcal{H}(d)$, and $\theta$
  and $\omega$ two elements of $F$.  Let $M$ and $N$ be two objects of
  $\cat{A}$, such that $\phi(M) = \phi(N) = d$.  Suppose that $M$ and
  $N$ are $\theta$-semistable, and that $M$ is $S_\theta$-equivalent
  to $N$.  Then, $M$ and $N$ are $\omega$-semistable, and $M$ is
  $S_\omega$-equivalent to $N$.
\end{proposition}

\begin{proof}
  Let $(M_i)_{i=0}^n$ and $(N_j)_{j=0}^m$ be Jordan-H\"older
  filtrations of $M$ and $N$, respectively, with respect to $\theta$.
  Then, by Proposition~\ref{pro:9}, $M$ and $N$ are
  $\omega$-semistable, and $(M_i)_{i=0}^n$ and $(N_j)_{j=0}^m$ are
  Jordan-H\"older filtrations of $M$ and $N$, respectively, with
  respect to $\omega$.  Now, because $M$ is $S_\theta$-equivalent to
  $N$, $n=m$, and there exists a permutation $\pi\in \sym{n}$, such
  that $M_{i-1}/M_i$ is isomorphic $N_{\pi(i)-1}/N_{\pi(i)}$ for every
  $i=1,\dotsc,n$.  As $(M_i)_{i=0}^n$ and $(N_j)_{j=0}^m$ are
  Jordan-H\"older filtrations of $M$ and $N$, respectively, with
  respect to $\omega$, it follows that $M$ is $S_\omega$-equivalent to
  $N$.
\end{proof}

\begin{remark}
  \label{rem:5}
  We use the following fact in the next proof.  Let $V$ be a
  finite-dimensional $\R$-vector space, and $V'$ a $\Q$-structure on
  $V$.  Let $(f_i)_{i\in I}$ be a family of $\R$-linear functions from
  $V$ to $\R$, which are $\Q$-rational.  Let
  $L=\bigcap_{i\in I}\Ker{f_i}$.  Then, the closure of $V'\cap L$ in
  $V$ equals $L$.
\end{remark}

\begin{proposition}
  \label{pro:12}
  Every facet of $\R^I$ with respect to the hyperplane arrangement
  $\mathcal{H}(d)$ contains an integral weight, that is, an element of
  $\Z^I$.
\end{proposition}

\begin{proof}
  Every element $e$ of $S_d$ is an element of $\N^I$, hence $f(d,e)$
  is $\Q$-rational.  Let $F$ be a facet of $\R^I$ with respect to
  $\mathcal{H}(d)$, and let $L=\supp{F}$.  Let $K$ denote the set of
  all elements $e$ of $S_d$ such that $F\subset H(d,e)$.  Then,
  \begin{equation*}
    \label{eq:62}
    L = \bigcap_{e\in K} H(d,e) = \Ker{f(d,e)},
  \end{equation*}
  hence, by Remark~\ref{rem:5}, the closure of $\Q^I\cap L$ in $\R^I$
  equals $L$.  Now, by Remark~\ref{rem:4}, $F$ is open in $L$.
  Therefore, there exists an open subset $U$ of $\R^I$, such that
  $F=U\cap L$.  Let $\theta$ be an element of the non-empty set $F$.
  Then, $\theta$ belongs to the closure $L$ of $\Q^I\cap L$ in $\R^I$,
  and $U$ is an open neighbourhood of $\theta$ in $\R^I$, hence there
  exists an element $\xi$ in $(\Q^I\cap L)\cap U = \Q^I\cap F$.  Let
  $n$ be a strictly positive integer such that $\omega = n\xi$ belongs
  to $\Z^I$.  We claim that $\omega\in F$.  To see this, let
  $e\in S_d$.  Then, $f(d,e)$ is $\R$-linear, hence
  $f(d,e)(\omega) = nf(d,e)(\xi)$.  As $n>0$, this implies that
  $\sgn{f(d,e)(\omega)} = \sgn{f(d,e)(\xi)}$.  Since $f(d,e)$ is a
  defining function of $H(d,e)$, it follows that
  $\omega \sim_{H(d,e)} \xi$.  As this is true for all $e\in S_d$,
  $\omega\sim \xi$, hence $\omega$ belongs to the facet $F$ of $\R^I$
  through $\xi$.  Thus, $\omega$ is an element of $\Z^I\cap F$.
\end{proof}

\begin{proposition}
  \label{pro:13}
  Let $d$ be a non-zero element of $\N^I$, and let $\theta \in \R^I$.
  Then, there exists an integral weight $\omega$ of $\cat{A}$, with
  the following properties:
  \begin{enumerate}
  \item \label{item:38} Any object $M$ of $\cat{A}$, such that
    $\phi(M) = d$, is $\theta$-semistable (respectively,
    $\theta$-stable, $\theta$-polystable) if and only if it is
    $\omega$-semistable (respectively, $\omega$-stable,
    $\omega$-polystable).
  \item \label{item:39} If $M$ is a $\theta$-semistable object of
    $\cat{A}$ such that $\phi(M)=d$, then every Jordan-H\"older
    filtration of $M$ with respect to $\theta$ is a Jordan-H\"older
    filtration of $M$ with respect to $\omega$, and conversely.
  \item \label{item:40} Two $\theta$-semistable objects $M$ and $N$ of
    $\cat{A}$, such that $\phi(M) = \phi(N) = d$, are
    $S_\theta$-equivalent if and only if they are
    $S_\omega$-equivalent.
  \end{enumerate}
\end{proposition}

\begin{proof}
  This Proposition follows immediately from Propositions
  \ref{pro:8}--\ref{pro:11} and \ref{pro:12}.
\end{proof}

\section{Representations of quivers}
\label{sec:repr-quiv}

In this section, we will specialise the constructs of the previous
sections to the specific abelian category of the representations of a
quiver over a field.  We begin by defining quivers and their
representations.  We then formulate the notion of the semistability of
a representation of a quiver with respect to a weight, as an instance
of the general theory described in Section~\ref{sec:semist-with-resp}.
In the last part of the Section, we look at special Hermitian metrics
on complex representations of a quiver.

\subsection{The category of representations}
\label{sec:categ-repr}

A \emph{quiver} $Q$ is a quadruple $(Q_0,Q_1,\sr,\tg)$, where $Q_0$
and $Q_1$ are sets, and $\sr: Q_1\to Q_0$, and $\tg: Q_1\to Q_0$ are
functions.  The elements of $Q_0$ are called the \emph{vertices of
  $Q$}, and those of $Q_1$ are called the \emph{arrows of $Q$}.  For
any arrow $\alpha$ of $Q$, the vertex $\sr(\alpha)$ is called the
\emph{source} of $\alpha$, and the vertex $\tg(\alpha)$ is called the
\emph{target} of $\alpha$.  If $\sr(\alpha)=a$ and $\tg(\alpha)=b$,
then we say that $\alpha$ is an arrow \emph{from} $a$ \emph{to} $b$,
and write $\alpha:a\to b$.  We say that $Q$ is \emph{vertex-finite} if
the set $Q_0$ is finite, \emph{arrow-finite} if the set $Q_1$ is
finite, and \emph{finite} if it is both vertex-finite and
arrow-finite.  The quiver $(\emptyset,\emptyset,\sr,\tg)$, where $\sr$
and $\tg$ are the empty functions, is called the \emph{empty} quiver.
We say that a quiver $Q$ is \emph{non-empty} if it is not equal to the
empty quiver, or equivalently, if the set $Q_0$ of its vertices is
non-empty.

Let $k$ be a field.  A \emph{representation of $Q$ over $k$} is a pair
$(V,\rho)$, where $V=(V_a)_{a\in Q_0}$ is a family of
finite-dimensional $k$-vector spaces, and
$\rho=(\rho_\alpha)_{\alpha\in Q_1}$ is a family of $k$-linear maps
$\rho_\alpha : V_{\sr(\alpha)} \to V_{\tg(\alpha)}$.  We will often
drop the base field $k$ from the terminology.  If $(V,\rho)$ and
$(W,\sigma)$ are two representations of $Q$, then a \emph{morphism
  from $(V,\rho)$ to $(W,\sigma)$} is a family $f=(f_a)_{a\in Q_0}$ of
$k$-linear maps $f_a:V_a\to W_a$, such that for every $\alpha\in Q_1$,
the diagram
\begin{equation*}
  \label{eq:63}
  \xymatrix{%
    V_{\sr(\alpha)} \ar[r]^{\rho_\alpha} \ar[d]_{f_{\sr(\alpha)}}
    & V_{\tg(\alpha)} \ar[d]^{f_{\tg(\alpha)}} \\
    W_{\sr(\alpha)} \ar[r]_{\sigma_\alpha}
    & W_{\tg(\alpha)}}
\end{equation*}
commutes.  If $(V,\rho)$, $(W,\sigma)$, and $(X,\tau)$ are three
representations of $Q$, $f$ a morphism from $(V,\rho)$ to
$(W,\sigma)$, and $g$ a morphism from $(W,\sigma)$ to $(X,\tau)$, then
the \emph{composite} of $f$ and $g$ is the family $g\circ f$ defined
by $g\circ f = (g_a\circ f_a)_{a\in Q_0}$.  It is easy to verify that
$g\circ f$ is a morphism from $(V,\rho)$ to $(X,\tau)$.  We thus get a
category $\Repr[k]{Q}$, whose objects are representations of $Q$ over
$k$, and whose morphisms are defined as above.

For any two representations $(V,\rho)$ and $(W,\sigma)$ of $Q$, the
set $\Hom{(V,\rho)}{(W,\sigma)}$ is a $k$-subspace of the $k$-vector
space $\bigoplus_{a\in Q_0} \Hom[k]{V_a}{W_a}$.  If $(X,\tau)$ is
another representation of $Q$, then the composition operator
\begin{equation*}
  \label{eq:64}
  \Hom{(W,\sigma)}{(X,\tau)} \times \Hom{(V,\rho)}{(W,\sigma)} \to
  \Hom{(V,\rho)}{(X,\tau)}
\end{equation*}
is $k$-bilinear.  Any representation $(V,\rho)$ such that the
$k$-vector space $V_a$ is zero for all $a\in Q_0$ is a zero object in
this category.  For every finite family $(V_i,\rho_i)_{i\in I}$ of
representations of $Q$, the pair $(V,\rho)$, which is defined by
\begin{equation*}
  \label{eq:65}
  V_a = \bigoplus_{i\in I}V_{i,a},
  \quad
  \rho_\alpha = \bigoplus_{i\in I}\rho_{i,\alpha},
\end{equation*}
for all $a\in Q_0$ and $\alpha\in Q_1$, is a coproduct of
$(V_i,\rho_i)_{i\in I}$ in $\Repr[k]{Q}$.  Thus, $\Repr[k]{Q}$ is a
$k$-linear additive category.

If $f: (V,\rho)\to (W,\sigma)$ is a morphism of representations of
$Q$, then the pair $(V',\rho')$, where $V'_a = \Ker{f_a}$ for all
$a\in Q_0$, and $\rho'_\alpha : V'_{\sr(\alpha)} \to V'_{\tg(\alpha)}$
is the restriction of $\rho_\alpha$ for each $\alpha\in Q_1$, is a
representation of $Q$, and there is an obvious morphism
$i : (V',\rho') \to (V,\rho)$, which is given by the inclusion maps
$i_a : V'_a \to V_a$ for all $a\in Q_0$.  The representation
$(V',\rho')$, together with the morphism $i$, is a kernel of $f$ in
$\Repr[k]{Q}$.  Similarly, the pair $(W',\sigma')$, where
$W'_a=\Coker{f_a}$ for all $a\in Q_0$, and
$\sigma'_\alpha : W'_{\sr(\alpha)} \to W'_{\tg(\alpha)}$ is the
$k$-linear map induced by $\sigma_\alpha$ for each $\alpha\in Q_1$, is
a representation of $Q$, and there is a morphism
$\pi : (W,\sigma) \to (W',\sigma')$, which is given by the canonical
projections $\pi_a : W_a \to W'_a$ for all $a\in Q_0$.  The
representation $(W',\sigma')$, together with the morphism $\pi$, is a
cokernel of $f$ in $\Repr[k]{Q}$.  Thus, every morphism in this
additive category has a kernel and a cokernel.  It is obvious that the
canonical morphism from $\Coker{i}$ to $\Ker{\pi}$ is an isomorphism.
It follows that $\Repr[k]{Q}$ is a $k$-linear abelian category.

A \emph{subrepresentation} of a representation $(V,\rho)$ of $Q$ is a
subobject of $(V,\rho)$ in the category $\Repr[k]{Q}$.  Every family
$(V_i,\rho_i)_{i\in I}$ of subrepresentations of $(V,\rho)$ has a meet
$\bigcap_{i\in I}(V_i,\rho_i)$, and a join
$\sum_{i\in I}(V_i,\rho_i)$.  Thus, the category $\Repr[k]{Q}$ has all
meets and joins of subobjects.

\begin{remark}
  \label{rem:1}
  For every representation $(V,\rho)$ of $Q$, and for any element $c$
  of $k$, we have a representation $(V,c\rho)$ of $Q$, where
  $c\rho = (c\rho_\alpha)_{\alpha\in Q_1}$.  If $(W,\sigma)$ is a
  subrepresentation of $(V,\rho)$, then $(W,c\sigma)$ is a
  subrepresentation of $(V,c\rho)$.  It is also obvious that if
  $(V_i,\rho_i)_{i\in I}$ is a finite family of representations of
  $Q$, then
  \begin{equation*}
    \label{eq:66}
    (\bigoplus V_i, c(\bigoplus_{i\in I}\rho_i)) =
    (\bigoplus V_i, \bigoplus_{i\in I}(c\rho_i)).
  \end{equation*}
  Lastly, if $f: (V,\rho) \to (W,\sigma)$ is a morphism of
  representations of $Q$, then $f$ is also a morphism of
  representations from $(V,c\rho)$ to $(W,c\sigma)$.
\end{remark}

\subsection{Semistability and stability of representations}
\label{sec:semist-stab-repr}

Fix a non-empty vertex-finite quiver $Q$, and a field $k$.  For every
representation $(V,\rho)$ of $Q$ over $k$, and $a\in Q_0$, let
\begin{equation*}
  \label{eq:67}
  \dim[a]{V,\rho} = \dim[k]{V_a}.
\end{equation*}
Then, $(\dim[a]{V,\rho})_{a\in Q_0}$ is a non-empty finite positive
family of additive functions from the abelian category $\Repr[k]{Q}$
to $\Z$.  Therefore, the statements of Section
\ref{sec:stab-with-resp} are applicable here.  The following are the
versions for representations of some of the notions defined there.

For every representation $(V,\rho)$, the element
\begin{equation*}
  \label{eq:68}
  \dim{V,\rho} = (\dim[k]{V_a})_{a\in Q_0}
\end{equation*}
of $\N^{Q_0}$ is called the dimension vector of $(V,\rho)$, and the
natural number
\begin{equation*}
  \label{eq:69}
  \rk{V,\rho} = \sum_{a\in Q_0}\dim[k]{V_a}
\end{equation*}
is called the rank of $(V,\rho)$.

An element of the $\R$-vector space $\R^{Q_0}$ is called a weight of
$Q$.  We say that a weight is rational (respectively, integral) if it
belongs to the subset $\Q^{Q_0}$ (respectively, $\Z^{Q_0}$) of
$\R^{Q_0}$.  We fix a weight $\theta$ of $Q$.

For any representation $(V,\rho)$ of $Q$, the $\theta$-degree of
$(V,\rho)$ is the real number
\begin{equation*}
  \label{eq:70}
  \deg_\theta(V,\rho) = \sum_{a\in Q_0}\theta_a \dim[k]{V_a}.
\end{equation*}
If $(V,\rho)\neq 0$, the real number
\begin{equation*}
  \label{eq:71}
  \mu_\theta(V,\rho) = \frac{\deg_\theta(V,\rho)}{\rk{V,\rho}},
\end{equation*}
is called the $\theta$-slope of $(V,\rho)$.

A representation $(V,\rho)$ of $Q$ is called $\theta$-semistable
(respectively, $\theta$-stable) if it is non-zero, and if
\begin{equation*}
  \label{eq:72}
  \mu_\theta(W,\sigma) \leq \mu_\theta(V,\rho)
  \quad
  (\text{respectively,}~ \mu_\theta(W,\sigma) < \mu_\theta(V,\rho))
\end{equation*}
for every non-zero proper subrepresentation $(W,\sigma)$ of
$(V,\rho)$.  We say that a representation of $Q$ is
$\theta$-polystable if it is $\theta$-semistable, and is isomorphic to
the direct sum of a finite family of $\theta$-stable representations
of $Q$.  There are obvious versions for representations of all the
results of Section~\ref{sec:stab-with-resp}.

\begin{remark}
  \label{rem:2}
  For any non-zero element $c$ of $k$, the $\theta$-semistability
  (respectively, $\theta$-stability, $\theta$-polystability) of a
  representation $(V,\rho)$ of $Q$ over $k$ is equivalent to the
  $\theta$-semistability (respectively, $\theta$-stability,
  $\theta$-polystability) of $(V,c\rho)$.  Also, if $\zeta$ is a
  strictly positive real number, and $\omega=\zeta\theta$, then a
  representation of $Q$ is $\theta$-semistable (respectively,
  $\theta$-stable, $\theta$-polystable) if and only if it is
  $\omega$-semistable (respectively, $\omega$-stable,
  $\omega$-polystable).
\end{remark}

\subsection{Einstein-Hermitian metrics on complex representations}
\label{sec:einst-herm-metr}

Let $Q$ be a non-empty finite quiver, and fix a weight $\theta$ of
$Q$.  All the representations of $Q$ considered in this subsection
will be over $\C$.

A \emph{Hermitian metric} on a representation $(V,\rho)$ of $Q$ is a
family $h=(h_a)_{a\in Q_0}$ of Hermitian inner products
$h_a:V_a\times V_a \to \C$.  Given a Hermitian metric $h$ on
$(V,\rho)$, for every vertex $a$ of $Q_0$, we have an endomorphism
$K_\theta(V,\rho)_a$ of the $\C$-vector space $V_a$, which is defined
by
\begin{equation*}
  \label{eq:73}
  K_\theta(V,\rho)_a =
  \theta_a \id{V_a} +
  \sum_{\alpha\in \tg^{-1}(a)} \rho_\alpha \circ \rho_\alpha^* -
  \sum_{\alpha\in \sr^{-1}(a)}\rho_\alpha^* \circ \rho_\alpha,
\end{equation*}
where, for each $\alpha\in Q_1$,
$\rho_\alpha^* : V_{\tg(\alpha)} \to V_{\sr(\alpha)}$ is the adjoint
of $\rho_\alpha : V_{\sr(\alpha)} \to V_{\tg(\alpha)}$ with respect to
the Hermitian inner products $h_{\sr(\alpha)}$ and $h_{\tg(\alpha)}$
on $V_{\sr(\alpha)}$ and $V_{\tg(\alpha)}$, respectively.  We say that
the metric $h$ is \emph{Einstein-Hermitian} with respect to $\theta$
if there exists a constant $c\in \C$, such that
\begin{equation*}
  \label{eq:74}
  K_\theta(V,\rho)_a = c \id{V_a}
\end{equation*}
for all $a\in Q_0$.  If this is the case, and if $(V,\rho)$ is
non-zero, then it is easy to see that $c=\mu_\theta(V,\rho)$, hence
\begin{equation*}
  \label{eq:75}
  \sum_{\alpha\in \tg^{-1}(a)} \rho_\alpha \circ \rho_\alpha^* -
  \sum_{\alpha\in \sr^{-1}(a)}\rho_\alpha^* \circ \rho_\alpha =
  (\mu_\theta(V,\rho) - \theta_a) \; \id{V_a}
\end{equation*}
for all $a\in Q_0$.  If we are considering more than one Hermitian
metric on $(V,\rho)$, and want to indicate the dependence of
$K_\theta(V,\rho)$ on the metric, we will write $K_\theta(V,\rho,h)$
instead of $K_\theta(V,\rho)$.

The following Proposition is a consequence of \cite[Proposition
6.5]{KMR}.  The restriction to rational weights here is due to the
fact that the cited result is proved in that reference only for
integral weights.

\begin{proposition}
  \label{pro:14}
  Let $\theta$ be a rational weight of $Q$, and $(V,\rho)$ a non-zero
  representation of $Q$.  Then, $(V,\rho)$ has an Einstein-Hermitian
  metric with respect to $\theta$ if and only if it is
  $\theta$-polystable.  Moreover, if $h_1$ and $h_2$ are two
  Einstein-Hermitian metrics on $(V,\rho)$ with respect to $\theta$,
  then there exists an automorphism $f$ of $(V,\rho)$, such that
  \begin{equation*}
    \label{eq:79}
    h_{1,a}(v,w) = h_{2,a}(f_a(v),f_a(w))
  \end{equation*}
  for all $a\in Q_0$ and $v,w\in V_a$.
\end{proposition}

Given a Hermitian metric $h$ on a representation $(V,\rho)$ of $Q$, we
say that two subrepresentations $(V_1,\rho_1)$ and $(V_2,\rho_2)$ of
$(V,\rho)$ are \emph{orthogonal} with respect to $h$ if for every
$a\in Q_0$, the subspaces $V_{1,a}$ and $V_{2,a}$ of $V_a$ are
orthogonal with respect to the Hermitian inner product $h_a$ on $V_a$.

\begin{corollary}
  \label{cor:1}
  Let $\theta$ be a rational weight of $Q$, $(V,\rho)$ a non-zero
  representation of $Q$, and $h$ an Einstein-Hermitian metric on
  $(V,\rho)$ with respect to $\theta$.  Then, there exists a finite
  family $(V_i,\rho_i)_{i\in I}$ of $\theta$-stable subrepresentations
  of $Q$, such that $(V,\rho) = \bigoplus_{i\in I}(V_i,\rho_i)$, and
  such that, for all $i,j\in I$ with $i\neq j$, $(V_i,\rho_i)$ and
  $(V_j,\rho_j)$ are orthogonal with respect to $h$.
\end{corollary}

\begin{proof}
  By Proposition~\ref{pro:14}, $(V,\rho)$ is $\theta$-polystable,
  hence there exists a finite family $(W_i,\sigma_i)_{i\in I}$ of
  $\theta$-stable subrepresentations of $Q$, such that
  $(V,\rho) = \bigoplus_{i\in I}(W_i,\sigma_i)$.  For each $i\in I$,
  there exists an Einstein-Hermitian metric $h_i$ on $(W_i,\sigma_i)$.
  Let $h'$ denote the Hermitian metric $\oplus_{i\in I} h_i$ on
  $(V,\rho)$.  Then, $h'$ is an Einstein-Hermitian metric on
  $(V,\rho)$.  Therefore, there exists an automorphism $f$ of
  $(V,\rho)$, such that $h'_a(v,w) = h_a(f_a(v),f_a(w))$ for all
  $a\in Q_0$ and $v,w\in V_a$.  Let $(V_i,\rho_i) = f((W_i,\sigma_i))$
  for every $i\in I$.  Then, $(V_i,\rho_i)_{i\in I}$ is the desired
  family of subrepresentations of $(V,\rho)$.
\end{proof}

Let $h$ be a Hermitian metric on $(V,\rho)$.  We say that an
endomorphism $f$ of $(V,\rho)$ is \emph{skew-Hermitian} with respect
to $h$ if for every $a\in Q_0$, the endomorphism $f_a$ of $V_a$ is
skew-Hermitian with respect to $h_a$, that is,
\begin{equation*}
  \label{eq:87}
  h_a(f_a(v),w) + h_a(v,f_a(w)) = 0
\end{equation*}
for all $v,w\in V_a$.  We denote the set of all skew-Hermitian
endomorphisms of $(V,\rho)$ with respect to $h$ by $\End{V,\rho,h}$.
It is an $\R$-subspace of the $\C$-vector space $\End{V,\rho}$.  We
say that $h$ is \emph{irreducible} if for every endomorphism $f$ of
$(V,\rho)$ that is skew-Hermitian with respect to $h$, there exists
$\lambda \in \C$, such that $f=\lambda\id{(V,\rho)}$.  The complex
number $\lambda$ is then purely imaginary, hence $h$ is irreducible if
and only if
\begin{equation*}
  \label{eq:88}
  \End{V,\rho,h} = \sqrt{-1} \, \R \, \id{(V,\rho)}.
\end{equation*}

The following result on irreducible representations can be proved
easily.
\begin{proposition}
  \label{pro:15}
  Let $(V,\rho)$ be a non-zero representation of $Q$, and $h$ a
  Hermitian metric on $(V,\rho)$.  Then, the following are equivalent:
  \begin{enumerate}
  \item \label{item:41} $h$ is irreducible.
  \item \label{item:42} If $(V_i,\rho_i)_{i\in I}$ is a finite family
    of subrepresentations of $(V,\rho)$,
    \begin{equation*}
      \label{eq:89}
      (V,\rho) = \bigoplus_{i\in I}(V_i,\rho_i),
    \end{equation*}
    and $(V_i,\rho_i)$ and $(V_j,\rho_j)$ are orthogonal with respect
    to $h$ for all $i,j\in I$ with $i\neq j$, then there exists
    $i\in I$, such that $(V_i,\rho_i)=(V,\rho)$.
  \end{enumerate}
\end{proposition}

\begin{proposition}
  \label{pro:16}
  Let $\theta$ be a rational weight of $Q$, $(V,\rho)$ a non-zero
  representation of $Q$, and $h$ an Einstein-Hermitian metric on
  $(V,\rho)$ with respect to $\theta$.  Then, the following are
  equivalent:
  \begin{enumerate}
  \item \label{item:43} $h$ is irreducible.
  \item \label{item:44} $(V,\rho)$ is $\theta$-stable.
  \item \label{item:45} $(V,\rho)$ is Schur.
  \end{enumerate}
\end{proposition}

\proof (\ref{item:43})\Limp(\ref{item:44}): Suppose $h$ is
irreducible.  Since $h$ is Einstein-Hermitian with respect to
$\theta$, by Corollary~\ref{cor:1}, there exists a finite family
$(V_i,\rho_i)_{i\in I}$ of $\theta$-stable subrepresentations of $Q$,
such that $(V,\rho) = \bigoplus_{i\in I}(V_i,\rho_i)$, and such that,
for all $i,j\in I$ with $i\neq j$, $(V_i,\rho_i)$ and $(V_j,\rho_j)$
are orthogonal with respect to $h$.  As $h$ is irreducible, by
Proposition~\ref{pro:15}, there exists $i\in I$ such that
$(V,\rho)=(V_i,\rho_i)$.  Therefore, $(V,\rho)$ is $\theta$-stable.

(\ref{item:44})\Limp(\ref{item:45}): Follows from
Proposition~\ref{pro:2}(\ref{item:24}).

(\ref{item:45})\Limp(\ref{item:43}): Suppose $f$ is a skew-Hermitian
endomorphism of $(V,\rho)$.  Then, by
Proposition~\ref{pro:2}(\ref{item:24})--(\ref{item:25}), there exists
$\lambda\in \C$ such that $f=\lambda\id{(V,\rho)}$.  Therefore, $h$ is
irreducible. \qed

\section{Families of representations}
\label{sec:famil-repr}

Fix a non-empty finite quiver $Q$.  We will consider only complex
representations of $Q$ in this section.

\subsection{Families parametrised by complex spaces}
\label{sec:famil-param-compl}

Let $T$ be a complex space.  By the \emph{Zariski topology} on $T$, we
mean the topology whose closed sets are the analytic subsets of $T$.
It is obviously coarser than the given topology on $T$, which we will
call the \emph{strong topology}.  In this context, the terms ``open'',
``continuous'', etc., without any qualifiers, are with respect to the
strong topology.  As usual, we denote the structure sheaf of $T$ by
$\struct{T}$.  If $E$ is any $\struct{T}$-module, then for any
$t\in T$, we denote by $E_t$ the $\sstalk{T}{t}$-module which is the
stalk of $E$ at $t$, and by $E(t)$ the fibre of $E$ at $t$, that is,
the $\C$-vector space $\C \otimes_{\sstalk{T}{t}} E_t$.  For any
element $\gamma$ of $E_t$, we denote the value of $\gamma$ at $t$,
that is, the canonical image of $\gamma$ in $E(t)$, by $\gamma(t)$.
If $U$ is an open neighbourhood of $t$, and $s\in E(U)$, we denote by
$s_t$ the germ of $s$ at $t$, that is, the canonical image of $s$ in
$E_t$, and by $s(t)$ the value of $s$ at $t$, that is, the
element$s_t(t)$ of $E(t)$.  If $f: E \to F$ is a morphism of
$\struct{T}$-modules, then for each $t\in T$, we have a canonical
$\sstalk{T}{t}$-linear map $f_t : E_t \to F_t$, and a canonical
$\C$-linear map $f(t) : E(t) \to F(t)$.  We will identify any
holomorphic vector bundle on $E$ with the $\struct{T}$-module of its
holomorphic sections.

A \emph{family of representations} of $Q$ parametrised by a complex
space $T$ is a pair $(V,\rho)$, where $V=(V_a)_{a\in Q_0}$ is a family
of holomorphic vector bundles on $T$, and
$\rho=(\rho_\alpha)_{\alpha\in Q_1}$ is a family of morphisms
$\rho_\alpha : V_{s(\alpha)} \to V_{t(\alpha)}$ of holomorphic vector
bundles on $T$.  There are obvious notions of a morphism
$f : (V,\rho) \to (W,\sigma)$ between two families of representations
of $Q$ parametrised by $T$, and the restriction
$\restrict{(E,\rho)}{U}$ of a family of representations of $Q$
parametrised by $T$ to an open subspace $U$ of $T$.  We thus get the
category $\Repr[T]{Q}$ of families of representations of $Q$
parametrised by $T$, and the sheaf $\HOM{(V,\rho)}{(W,\sigma)}$ of
morphisms between two families of representations of $Q$ parametrised
by $T$.  The latter is an $\struct{T}$-submodule of the
$\struct{T}$-module $\oplus_{a\in Q_0}\HOM[\struct{T}]{V_a}{W_a}$.
For every morphism $f : (V,\rho) \to (W,\sigma)$ of families of
representations of $Q$ parametrised by $T$, and open subset $U$ of
$T$, we have an obvious restriction
$\restrict{f}{U} : \restrict{(V,\rho)}{U} \to
\restrict{(W,\sigma)}{U}$.

\begin{remark}
  \label{rem:3}
  Given two families of representations $(V,\rho)$ and $(W,\sigma)$ of
  $Q$ parametrised by $T$, we define two $\struct{T}$-modules $E$ and
  $F$ by
  \begin{equation*}
    \label{eq:102}
    E = \oplus_{a\in Q_0}\HOM[\struct{T}]{V_a}{W_a}, \quad
    F = \oplus_{\alpha\in Q_1}
    \HOM[\struct{T}]{V_{s(\alpha)}}{W_{t(\alpha)}},
  \end{equation*}
  and a morphism $u: E \to F$ of $\struct{T}$-modules by
  \begin{equation*}
    \label{eq:103}
    u_U(f) =
    (f_{\tg(\alpha)} \circ \restrict{\rho_\alpha}{U} -
    \restrict{\sigma_\alpha}{U} \circ f_{\sr(\alpha)})_{\alpha\in Q_1},
  \end{equation*}
  for every open subset $U$ of $T$, and for every $f=(f_a)_{a\in Q_0}$
  in
  \begin{equation*}
    \label{eq:104}
    E(U) = \oplus_{a\in Q_0}
    \Hom{\restrict{V_a}{U}}{\restrict{W_a}{U}}.
  \end{equation*}
  By the definition of $u$, we have
  $(\Ker{u})(U) =
  \Hom{\restrict{(V,\rho)}{U}}{\restrict{(W,\sigma)}{U}}$ for every
  open subset $U$ of $T$, hence
  $\HOM{(V,\rho)}{(W,\sigma)} = \Ker{u}$.  The assumption that $Q$ is
  finite, and the fact that $V_a$ is locally free for every
  $a\in Q_0$, imply that $E$ and $F$ are locally free, hence coherent,
  $\struct{T}$-modules.  It follows that the $\struct{T}$-module
  $\HOM{(V,\rho)}{(W,\sigma)}$ is coherent.
\end{remark}

Let $f : T' \to T$ be a morphism of complex spaces, and $(V,\rho)$ a
family of representations of $Q$ parametrised by $T$.  For each
$a\in Q_0$, define a holomorphic vector bundle $M_a$ on $T'$, by
$M_a = \pullback{f}{V_a}$.  Then, for each $\alpha\in Q_1$, we have a
morphism
$\phi_\alpha = \pullback{f}{\rho_\alpha} : M_{s(\alpha)} \to
M_{t(\alpha)}$ of holomorphic vector bundles on $T'$.  We thus, get a
family $(M,\phi)$ of representations of $Q$ parametrised by $T'$.  We
call it the \emph{pullback} of $(V,\rho)$ by $f$, and will denote it
by $\pullback{f}{V,\rho}$.  In particular, if $U$ is an open subspace
of $T$, and if $f: U \to T$ is the canonical morphism of ringed
spaces, then $\pullback{f}{V,\rho}$ is canonically isomorphic to the
restriction $\restrict{(V,\rho)}{U}$.  If $A$ is any complex subspace
of $T$, and $f: A \to T$ the canonical morphism, we will call
$\pullback{f}{V,\rho}$ the \emph{restriction} of $(V,\rho)$ to $A$,
and will denote it by $\restrict{(V,\rho)}{A}$.

Suppose $(V,\rho)$ is a family of representations of $Q$ parametrised
by $T$.  Then, for each point $t\in T$, we get a representation
$(V(t),\rho(t))$ of $Q$ over $\C$, which is defined by
$V(t)=(V_a(t))_{a\in Q_0}$ and
$\rho(t)=(\rho_\alpha(t))_{\alpha\in Q_1}$.  If $P$ is any property of
representations of $Q$ over an arbitrary field, we say that $(V,\rho)$
\emph{has the property} $P$ if for every $t\in T$, the representation
$(V(t),\rho(t))$ of $Q$ over $\C$ has the property $P$.  We can thus
speak of a \emph{family of non-zero representations} of $Q$
parametrised by $T$, a \emph{family of Schur representations} of $Q$
parametrised by $T$, etc.  If $\theta$ is a weight in $\R^{Q_0}$, we
can speak of a \emph{family of $\theta$-stable representations} of $Q$
parametrised by $T$, a \emph{family of $\theta$-semistable
  representations} of $Q$ parametrised by $T$, etc.

\begin{proposition}
  \label{pro:17}
  Let $T$ be a complex space, $(V,\rho)$ and $(W,\sigma)$ two families
  of representations of $Q$ parametrised by $T$, and $u: E\to F$ the
  morphism of $\struct{T}$-modules defined in Remark~\ref{rem:3}.
  Then, for every $t\in T$, there is a canonical isomorphism
  \begin{equation*}
    \label{eq:135}
    \Hom{(V(t),\rho(t))}{(W(t),\sigma(t))} \cong \Ker{u(t)}
  \end{equation*}
  of $\C$-vector spaces.  In particular, the function
  \begin{equation*}
    \label{eq:137}
    t \mapsto
    \dim[\C]{\Hom{(V(t),\rho(t))}{(W(t),\sigma(t))}} :
    T \to \N
  \end{equation*}
  is upper semi-continuous with respect to the Zariski topology on
  $T$.
\end{proposition}

\proof For every point $t\in T$, we have canonical identifications
\begin{equation*}
  \label{eq:4}
  E(t) = \oplus_{a\in Q_0} \Hom[\C]{V_a(t)}{W_a(t)},
  \quad
  F(t) = \oplus_{\alpha\in Q_1}
  \Hom[\C]{V_{\sr(\alpha)}(t)}{W_{\tg(\alpha)}(t)},
\end{equation*}
since the $\struct{T}$-modules $V_a$ and $W_a$ are locally free for
every $a\in Q_0$.  Under these identifications, the $\C$-linear map
$u(t) : E(t) \to F(t)$ takes any element $f = (f_a)_{a\in Q_0}$ of
$E(t)$ to the element
$(f_{\tg(\alpha)} \circ \rho_\alpha(t) - \sigma_\alpha(t) \circ
f_{\sr(\alpha)})_{\alpha\in Q_1}$ of $F(t)$.  Therefore, we have a
canonical $\C$-isomorphism
$\Ker{u(t)} \cong \Hom{(V(t),\rho(t))}{(W(t),\sigma(t))}$.  But, as
both $E$ and $F$ are locally free, the function
$t\mapsto \dim[\C]{\Ker{u(t)}}$ from $T$ to $\N$ is upper
semi-continuous with respect to the Zariski topology on $T$.  \qed

\begin{corollary}
  \label{cor:2}
  Let $T$ be a complex space, and $(V,\rho)$ a family of
  representations of $Q$ parametrised by $T$.  Then, the subset
  \begin{equation*}
    \label{eq:144}
    \set{t\in T \suchthat
      \dim[\C]{\End{V(t),\rho(t)}} \neq 1}.
  \end{equation*}
  of $T$ is analytic.  
\end{corollary}

\begin{corollary}
  \label{cor:5}
  Let $T$ be a complex space, and $(V,\rho)$ and $(W,\sigma)$ two
  families of representations of $Q$ parametrised by $T$.  Suppose
  that $T$ is reduced, and that the function
  \begin{equation*}
    \label{eq:165}
    t\mapsto
    \dim[\C]{\Hom{(V(t),\rho(t))}{(W(t),\sigma(t))}} :
    T\to \N
  \end{equation*}
  is locally constant.  Then, the $\struct{T}$-module
  $\HOM{(V,\rho)}{(W,\sigma)}$ is locally free.  Moreover, for every
  $t\in T$, there is a canonical $\C$-isomorphism
  \begin{equation*}
    \label{eq:166}
    (\HOM{(V,\rho)}{(W,\sigma)})(t) \cong
    \Hom{(V(t),\rho(t))}{(W(t),\sigma(t))}.
  \end{equation*}
\end{corollary}

\proof Let $u: E\to F$ be the morphism of $\struct{T}$-modules defined
in Remark~\ref{rem:3}.  Then, $\HOM{(V,\rho)}{(W,\sigma)} = \Ker{u}$.
Moreover, by Proposition~\ref{pro:17}, for every $t\in T$, the
$\C$-vector spaces $\Hom{(V(t),\rho(t))}{(W(t),\sigma(t))}$ and
$\Ker{u(t)}$ are canonically isomorphic.  Therefore, the function
$t\mapsto \dim[\C]{\Ker{u(t)}}$ from $T$ to $\N$ is locally constant.
Since the $\struct{T}$-modules $E$ and $F$ are locally free, and $T$
is reduced, this implies that the $\struct{T}$-module $\Ker{u}$ is
locally free, and that for every $t\in T$, there is a canonical
$\C$-isomorphism $(\Ker{u})(t) \cong \Ker{u(t)}$.  \qed

\begin{corollary}
  \label{cor:6}
  Let $(V,\rho)$ be a family of representations of $Q$ parametrised by a
  complex space $T$.  Then, the subset of $T$, consisting of all the
  points $t\in T$ such that the representation $(V(t),\rho(t))$ of $Q$
  over $\C$ is Schur, is open with respect to the Zariski topology on
  $T$.
\end{corollary}

\proof Let $\schur{T}$ denote the said subset of $T$.  By
Proposition~\ref{pro:2}(\ref{item:25}), $\schur{T}$ equals the set of
all the points $t\in T$, such that $\dim[\C]{\End{V(t),\rho(t)}}=1$.
Therefore, by Corollary~\ref{cor:2}, $T\setminus \schur{T}$ is an
analytic subset of $T$, and is hence Zariski closed.  \qed

\begin{corollary}
  \label{cor:7}
  Let $f_1 : S \to T_1$ and $f_2 : S \to T_2$ be morphisms of complex
  analytic spaces, $(V,\rho)$ a family of representations of $Q$
  parametrised by $T_1$, and $(W,\sigma)$ a family of representations
  of $Q$ parametrised by $T_2$.  Then, the function
  \begin{equation*}
    \label{eq:172}
    s \mapsto
    \dim[\C]{%
      \Hom{(V(f_1(s)),\rho(f_1(s)))}{(W(f_2(s)),\sigma(f_2(s)))}} :
    S \to \N
  \end{equation*}
  is upper semi-continuous with respect to the Zariski topology on
  $S$.
\end{corollary}

\proof Let $(M,\phi) = \pullback{f_1}{V,\rho}$ and
$(N,\psi) = \pullback{f_2}{W,\sigma}$.  Then, for every $s\in S$, we
have canonical isomorphisms
\begin{equation*}
  \label{eq:173}
  (M(s),\phi(s)) \cong (V(f_1(s)),\rho(f_1(s))), \quad
  (N(s),\psi(s)) \cong (W(f_2(s)),\sigma(f_2(s)))
\end{equation*}
of representations of $Q$ over $\C$, hence we get a canonical
$\C$-isomorphism
\begin{equation*}
  \label{eq:174}
  \Hom{(M(s),\phi(s))}{(N(s),\psi(s))} \cong
  \Hom{(V(f_1(s)),\rho(f_1(s)))}{(W(f_2(s)),\sigma(f_2(s)))}.
\end{equation*}
The Corollary now follows from Proposition~\ref{pro:17}.  \qed

\subsection{The Hausdorff property}
\label{sec:hausdorff-property}

Let $R$ be an equivalence relation on a topological space $T$.  We say
that two points $t_1$ and $t_2$ in $T$ are \emph{separated} with
respect to $R$ if there exist an open neighbourhood $U_1$ of $t_1$,
and an open neighbourhood $U_2$ of $t_2$, in $T$, such that
$U_1\cap U_2=\emptyset$, and both $U_1$ and $U_2$ are saturated with
respect to $R$.  This is equivalent to the condition that there exist
an open neighbourhood $U_1'$ of $\pi(t_1)$, and an open neighbourhood
$U_2'$ of $\pi(t_2)$, in $T'$, such that $U_1'\cap U_2'=\emptyset$,
where $T'=T/R$ is the quotient topological space of $T$ by $R$,
$\pi : T \to T'$ the canonical projection.  We say that $R$ is
\emph{open} if the saturation with respect to $R$ of any open subset
of $T$ is open in $T$.  This is equivalent to the condition that
$\pi: T \to T'$ is an open map.

\begin{remark}
  \label{rem:6}
  It is easy to verify the following facts:
  \begin{enumerate}
  \item \label{item:58} The closure of $R$ in $T\times T$ equals the
    set of all points $(t_1,t_2)$ in $T\times T$ such that $t_1$ and
    $t_2$ are not separated with respect to $R$.
  \item \label{item:59} The quotient topological space $T'=T/R$ is
    Hausdorff if and only if $R$ is closed in $T\times T$.
  \end{enumerate}
\end{remark}

Let $T$ be a complex space, and $(V,\rho)$ a family of representations
of $Q$ parametrised by $T$.  Define a relation $R$ on $T$ by setting
$t_1R t_2$ if the representations $(V(t_1),\rho(t_1))$ and
$(V(t_2),\rho(t_2))$ of $Q$ over $\C$ are isomorphic.  This is an
equivalence relation on $T$.  We will call it the equivalence relation
on $T$ \emph{induced} by $(V,\rho)$.

\begin{lemma}
  \label{lem:1}
  Let $T$ be a complex space, $(V,\rho)$ a family of non-zero
  representations of $Q$ parametrised by $T$, and $R$ the equivalence
  relation on $T$ induced by $(V,\rho)$.  Let $Z$ denote the closure
  of $R$ with respect to the Zariski topology on the product complex
  space $T\times T$.  Then, for every point $(t_1,t_2) \in Z$, there
  exist non-zero morphisms
  \begin{equation*}
    \label{eq:207}
    f : (V(t_1),\rho(t_1)) \to (V(t_2),\rho(t_2)), \quad
    g : (V(t_2),\rho(t_2)) \to (V(t_1),\rho(t_1))
  \end{equation*}
  of representations of $Q$ over $\C$.
\end{lemma}

\proof Let
\begin{equation*}
  \label{eq:208}
  A_1 =
  \set{(t_1,t_2) \in T\times T \suchthat
    \Hom{(V(t_1),\rho(t_1))}{(V(t_2),\rho(t_2))} \neq 0},
\end{equation*}
and
\begin{equation*}
  \label{eq:209}
  A_2 =
  \set{(t_1,t_2) \in T\times T \suchthat
    \Hom{(V(t_2),\rho(t_2))}{(V(t_1),\rho(t_1))} \neq 0}.
\end{equation*}
Then, it is obvious that $R$ is a subset of $A_1\cap A_2$.  On the
other hand, by Corollary~\ref{cor:7}, $A_1$ and $A_2$ are closed in
the Zariski topology on $T\times T$, and hence so is $A_1\cap A_2$.
It follows that $Z\subset A_1\cap A_2$.  \qed

\begin{proposition}
  \label{pro:26}
  Let $T$ be a complex space, $\theta: T \to \R^{Q_0}$ a continuous
  function, and $(V,\rho)$ a family of representations of $Q$
  parametrised by $T$.  Suppose that the equivalence relation $R$ on
  $T$ induced by $(V,\rho)$ is open, that the representation
  $(V(t),\rho(t))$ over $\C$ is $\theta(t)$-stable for every $t\in T$,
  and that the function $\theta$ is $R$-invariant.  Then, the quotient
  topological space $T/R$ is Hausdorff.
\end{proposition}

\proof By Remark~\ref{rem:6}(\ref{item:59}), it suffices to prove that
$R$ is strongly closed in $T\times T$.  Let $(t_1,t_2)$ be any point
in the strong closure $F$ of $R$ in $T\times T$.  Obviously, $F$ is
contained in the Zariski closure of $R$ in $T\times T$, hence, by
Lemma~\ref{lem:1}, there exists a non-zero morphism
$f : (V(t_1),\rho(t_1)) \to (V(t_2),\rho(t_2))$ of representations of
$Q$ over $\C$.  Now, for every $a\in Q_0$, the rank function
$t\mapsto \dim[\C]{V_a(t)} : T \to \N$ of the vector bundle $V_a$ is
locally constant.  Therefore, as $\theta$ is continuous, the function
$\phi : T \to \N^{Q_0} \times \R^{Q_0}$, which is defined by
$\phi(t) = (\dim{V(t),\rho(t)},\theta(t))$, is continuous.  Thus, the
set
\begin{equation*}
  \label{eq:217}
  G =
  \set{(t,t')\in T\times T \suchthat \phi(t) = \phi(t')}
\end{equation*}
is strongly closed in $T\times T$.  As $\theta$ is $R$-invariant,
$R\subset G$.  Therefore, $F\subset G$, hence $\phi(t_1)=\phi(t_2)$.
This implies that
$\mu_\theta(V(t_1),\rho(t_1)) = \mu_\theta(V(t_2),\rho(t_2))$, where
$\theta = \theta(t_1) =\theta(t_2)$.  Since $(V(t),\rho(t))$ is
$\theta(t)$-stable for all $t\in T$, by
Proposition~\ref{pro:2}(\ref{item:22}), we see that $f$ is an
isomorphism.  Therefore, $(t_1,t_2)\in R$.  This proves that $R$ is
strongly closed in $T\times T$.  \qed

\section{The moduli space of Schur representations}
\label{sec:moduli-space-schur}

\subsection{Quotient premanifolds}
\label{sec:quot-prem}

By a \emph{complex premanifold}, we mean a complex manifold without
any separation or countability conditions, that is, a topological
space with a maximal holomorphic atlas.  We use the term \emph{complex
  manifold} for a complex premanifold whose underlying topological
space is Hausdorff.

Let $R$ be an equivalence relation on a complex premanifold $X$, $Y$
the quotient topological space $X/R$, and $p: X \to Y$ the canonical
projection.  It is a theorem of Godement that the following statements
are equivalent:
\begin{enumerate}
\item \label{item:62} There exists a structure of a complex
  premanifold on $Y$ with the property that $p$ is a holomorphic
  submersion.
\item \label{item:63} The relation $R$ is a subpremanifold of
  $X\times X$, and the restricted projection $\pr{1} : R \to X$ is a
  submersion.
\end{enumerate}
Moreover, in that case, such a complex premanifold structure on $Y$ is
unique \cite[Part~II, Chapter~III, \S~12, Theorems 1--2]{SLG}.

We will use the above theorem in the context of group actions.  Let
$X$ be a topological space, and $G$ a topological group.  Suppose that
we are given a continuous right action of $G$ on $X$.  Let $R$ denote
the equivalence relation on $X$ defined by this action, and
$\tau : X\times G \to R$ the map $(x,g) \mapsto (x,xg)$.  For any two
subsets $A$ and $B$ of $X$, let
\begin{equation*}
  \label{eq:219}
  P_G(A,B) =
  \set{g\in G \suchthat Ag\cap B \neq \emptyset}.
\end{equation*}
If the action of $G$ on $X$ is free, then for every $(x,y)\in R$,
there exists a unique element $\phi(x,y)$ of $G$, such that
$y=x\phi(x,y)$; we thus get a map $\phi : R \to G$, which is called
the \emph{translation} map of the given action.  We say that the
action of $G$ on $X$ is \emph{principal} if it is free, and its
translation map is continuous.

\begin{remark}
  \label{rem:7}
  It is easy to verify the following assertions:
  \begin{enumerate}
  \item \label{item:64} The action of $G$ on $X$ is free if and only
    if the map $\tau$ is injective.
  \item \label{item:65} The following statements are equivalent:
    \begin{enumerate}
    \item \label{item:66} The action of $G$ on $X$ is principal.
    \item \label{item:67} The action of $G$ on $X$ is free, and its
      translation map is continuous at $(x,x)$ for all $x\in X$.
    \item \label{item:68} The action of $G$ on $X$ is free, and for
      every point $x\in X$, and for every neighbourhood $V$ of the
      identity element $e$ of $G$, there exists a neighbourhood $U$ of
      $x$ in $X$, such that $P_G(U,U)\subset V$.
    \item \label{item:69} The map $\tau$ is a homeomorphism.
    \end{enumerate}
  \end{enumerate}
\end{remark}

\begin{lemma}
  \label{lem:8}
  Let $X$ be a complex premanifold, and $G$ a complex Lie group.
  Suppose that we are given a principal holomorphic right action of
  $G$ on $X$.  Let $R$ be the equivalence relation on $X$ defined by
  the action of $G$, and $\tau: X \times G \to R$ the map
  $(x,g)\mapsto (x,xg)$.  Then, $R$ is a complex subpremanifold of
  $X\times X$, and $\tau$ is a biholomorphism.
\end{lemma}

\proof Let $\sigma : X\times G \to X\times X$ be the map
$(x,g)\mapsto (x,xg)$.  Thus, $\sigma(X\times G)=R$, and
$\tau : X \times G \to R$ is the map induced by $\sigma$.  Since the
action of $G$ on $X$ is principal, by Remark~\ref{rem:7}, the map
$\tau$ is a homeomorphism.  The map $\sigma$ is obviously holomorphic.
As the action of $G$ on $X$ is free, $\sigma$ is an immersion.
Therefore, $\sigma$ is a holomorphic embedding, its image $R$ is a
complex subpremanifold of $X\times X$, and $\tau$ is a biholomorphism.
\qed

\begin{lemma}
  \label{lem:10}
  Let $p : X \to Y$ be a surjective holomorphic submersion of complex
  premanifolds, and $G$ a complex Lie group.  Suppose that we are
  given a principal holomorphic right action of $G$ on $X$, such that
  $p^{-1}(p(x)) = xG$ for all $x\in X$.  Then, this action makes $p$ a
  holomorphic principal $G$-bundle.
\end{lemma}

\proof Let $R$ be the equivalence relation on $X$ defined by the
action of $G$, and $\tau: X \times G \to R$ the map
$(x,g)\mapsto (x,xg)$.  Then, by Lemma~\ref{lem:8}, $R$ is a complex
subpremanifold of $X\times X$, and $\tau$ is a biholomorphism.  Let
$b\in Y$.  As $p$ is surjective, there exists a point
$a\in p^{-1}(b)$.  Since $p$ is a submersion at $a$, there exist an
open neighbourhood $V$ of $b$ in $Y$, and a holomorphic section
$s : V \to X$ of $p$, such that $s(b)=a$.  The hypothesis on the
fibres of $p$ implies that the map $(c,g)\mapsto s(c)g$ is a
$G$-equivariant holomorphic bijection $u$ from $V\times G$ onto
$p^{-1}(V)$.  Its inverse is the composite
\begin{equation*}
  \label{eq:225}
  p^{-1}(V) \xrightarrow{\alpha}
  (p^{-1}(V) \times p^{-1}(V))\cap R \xrightarrow{\beta}
  p^{-1}(V) \times G \xrightarrow{\gamma}
  V \times G,
\end{equation*}
where
\begin{equation*}
  \label{eq:226}
  \alpha(x) = (x,s(p(x))), \quad
  \beta(y,z) = \tau^{-1}(y,z), \quad
  \gamma(x,g) = (p(x),g)
\end{equation*}
for all $x\in p^{-1}(V)$,
$(y,z)\in (p^{-1}(V)\times p^{-1}(V))\cap R$, and $g\in G$.  Since
$\tau^{-1}$ is holomorphic, $u^{-1}$ is also holomorphic.  By
definition, $p(u(c,g)) = c$ for all $c\in V$ and $g\in G$.  Thus, $u$
is a local trivialisation of $p$ at $b$.  It follows that $p$ is a
holomorphic principal $G$-bundle.  \qed

\begin{remark}
  \label{rem:8}
  The proof of Lemma~\ref{lem:10} also works to show that if
  $p : X \to Y$ is a surjective smooth submersion of smooth
  premanifolds, and $G$ a real Lie group, and if we are given a
  principal smooth right action of $G$ on $X$, such that
  $p^{-1}(p(x)) = xG$ for all $x\in X$, then this action makes $p$ a
  smooth principal $G$-bundle.
\end{remark}

\begin{proposition}
  \label{pro:27}
  Let $X$ be a complex premanifold, and $G$ a complex Lie group.
  Suppose that we are given a principal holomorphic right action of
  $G$ on $X$.  Let $Y$ be the quotient topological space $X/G$, and
  $p: X \to Y$ the canonical projection.  Then, there exists a unique
  structure of a complex premanifold on $Y$ such that $p$ is a
  holomorphic submersion.  This structure makes $p$ a holomorphic
  principal $G$-bundle.
\end{proposition}

\proof Let $R$ be the equivalence relation on $X$ defined by the
action of $G$, and $\tau: X \times G \to R$ the map
$(x,g)\mapsto (x,xg)$.  Then, by Lemma~\ref{lem:8}, $R$ is a complex
subpremanifold of $X\times X$, and $\tau$ is a biholomorphism.  Since
$\pr{1}\circ \tau = \pr{1}$, and $\pr{1} : X\times G\to X$ is clearly
a submersion, it follows that $\pr{1} : R \to X$ is a submersion.
Therefore, by Godement's theorem, there exists a unique structure of a
complex premanifold on $Y$, such that $p$ is a holomorphic submersion.
It is obvious that $p$ is surjective, and that $p^{-1}(p(x)) = xG$ for
all $x\in X$.  Therefore, by Lemma~\ref{lem:10}, $p$ is a holomorphic
principal $G$-bundle.  \qed

Let $G$ be a complex Lie group acting holomorphically on the right of
a complex premanifold $X$, $Y$ the quotient topological space $X/G$,
and $p : X \to Y$ the canonical projection.  Let $H$ be a normal
complex Lie subgroup of $G$, $\bar{G}$ the complex Lie group
$H\backslash G$, and $\pi: G \to \bar{G}$ the canonical projection.
If the stabiliser $G_x$ of any point $x\in X$ equals $H$, then there
is an induced holomorphic right action of $\bar{G}$ on $X$.

\begin{corollary}
  \label{cor:3}
  Suppose that the stabiliser $G_x$ of any point $x\in X$ equals $H$,
  and that for each $x\in X$, and $H$-invariant neighbourhood $V$ of
  $e$ in $G$, there exists a neighbourhood $U$ of $x$ in $X$, such
  that $P_G(U,U)\subset V$.  Then, the action of $\bar{G}$ on $X$ is
  principal, and there exists a unique structure of a complex
  premanifold on $Y$, such that $p$ is a holomorphic submersion.
  Moreover, with the induced action of $\bar{G}$ on $X$, $p$ is a
  holomorphic principal $\bar{G}$-bundle.
\end{corollary}

\proof The induced action of $\bar{G}$ on $X$ is free, since $G_x=H$
for all $x\in X$.  Let $x\in X$, and let $W$ be a neighbourhood of $e$
in $\bar{G}$.  Then, $V=\pi^{-1}(W)$ is an $H$-invariant neighbourhood
of $e$ in $G$.  Therefore, by hypothesis, there exists an open
neighbourhood $U$ of $x$ in $X$, such that $P_G(U,U)\subset V$.  Now,
$U = U\cap X$ is an open neighbourhood of $x$ in $X$, and
$P_{\bar{G}}(U,U) \subset \pi(P_G(U,U)) \subset \pi(P_G(U,U)) \subset
\pi(V) \subset W$.  Therefore, by Remark~\ref{rem:7}, the action of
$\bar{G}$ on $X$ is principal.  It is obvious that the induced action
of $\bar{G}$ on $X$ is holomorphic, that $X/\bar{G} = X/G = Y$, and
that the canonical projection from $X$ to $X/\bar{G}$ equals $p$.  The
Corollary now follows from Proposition~\ref{pro:27}.  \qed

\subsection{The complex premanifold of Schur representations}
\label{sec:compl-prem-schur}

Let $Q$ be a non-empty finite quiver.  We will consider only complex
representations of $Q$ in this subsection.  Let $d=(d_a)_{a\in Q_0}$
be a non-zero element of $\N^{Q_0}$, and fix a family
$V=(V_a)_{a\in Q_0}$ of $\C$-vector spaces, such that
$\dim[\C]{V_a}=d_a$ for all $a\in Q_0$.

Denote by $\mathcal{A}$ the finite-dimensional $\C$-vector space
$\bigoplus_{\alpha\in Q_1}\Hom[\C]{V_{\sr(\alpha)}}{V_{\tg(\alpha)}}$.
For every element $\rho$ of $\mathcal{A}$, we have a representation
$(V,\rho)$ of $Q$.  Moreover, for every representation $(W,\sigma)$ of
$Q$, such that $\dim{W,\sigma}=d$, there exists an element $\rho$ of
$\mathcal{A}$, such that the representations $(V,\rho)$ and
$(W,\sigma)$ are isomorphic.

We give the vector space $\mathcal{A}$ the usual topology, and the
usual structure of a complex manifold.  For each $a\in Q_0$, denote by
$E_a$ the trivial holomorphic vector bundle $\mathcal{A}\times V_a$ on
$\mathcal{A}$.  Then, for every $\alpha\in Q_1$, we have a morphism
$\theta_\alpha : E_{\sr(\alpha)} \to E_{\tg(\alpha)}$ of holomorphic
vector bundles, which is defined by
$\theta_\alpha(\rho,v) = (\rho,\rho_\alpha(v))$ for all
$(\rho,v)$ in $E_{\sr(\alpha)}$.  We thus get a family $(E,\theta)$ of
representations of $Q$ parametrised by $\mathcal{A}$, where
$E=(E_a)_{a\in Q_0}$ and $\theta=(\theta_\alpha)_{\alpha\in Q_1}$.  By
definition, for each point $\rho\in \mathcal{A}$, the fibre
representation $E(\rho)$ is precisely $(V,\rho)$.

Let $G$ be the complex Lie group $\prod_{a\in Q_0}\Aut[\C]{V_a}$.
There is a canonical holomorphic linear right action
$(\rho,g)\mapsto \rho g$ of $G$ on $\mathcal{A}$, which is defined by
\begin{equation*}
  \label{eq:231}
  (\rho g)_\alpha =
  g_{\tg(\alpha)}^{-1} \circ \rho_\alpha \circ g_{\sr(\alpha)}
\end{equation*}
for all $\rho\in \mathcal{A}$, $g\in G$, and $\alpha\in Q_1$.  For all
$\rho,\sigma \in \mathcal{A}$ and $g\in G$, we have $\sigma = \rho g$
if and only if $g$ is an isomorphism of representations of $Q$, from
$(V,\sigma)$ to $(V,\rho)$.  In other words, two points $\rho$ and
$\sigma$ of $\mathcal{A}$ lie on the same orbit of $G$ if and only if
the representations $(V,\rho)$ and $(V,\sigma)$ of $Q$ are isomorphic.
Thus, the map which takes every point $\rho$ of $\mathcal{A}$ to the
representation $(V,\rho)$ induces a bijection from the quotient set
$\mathcal{A}/G$ onto the set of isomorphism classes of representations
$(W,\sigma)$ of $Q$, such that $\dim{W,\sigma} = d$.

Denote by $H$ the central complex Lie subgroup of $G$ consisting of
all elements of the form $ce$, as $c$ runs over $\units{\C}$, where
$e=(\id{V_a})_{a\in Q_0}$ is the identity element of $G$.  Let
$\bar{G}$ denote the complex Lie group $H\backslash G$,
$\pi : G \to \bar{G}$ the canonical projection.  Define $\mathcal{B}$
to be the set of all points $\rho$ of $\mathcal{A}$, such that the
representation $(V,\rho)$ of $Q$ is Schur.  It is a $G$-invariant
subset of $\mathcal{A}$.  By Proposition~\ref{pro:2}(\ref{item:25}), a
point $\rho$ of $\mathcal{A}$ lies in $\mathcal{B}$ if and only if its
stabiliser $G_\rho$ equals $H$.  Corollary~\ref{cor:6}, applied to the
family $(E,\theta)$ of representations of $Q$ parametrised by
$\mathcal{A}$, implies that $\mathcal{B}$ is Zariski open in
$\mathcal{A}$, and is hence an open complex submanifold of
$\mathcal{A}$.  Let $M$ denote the quotient topological space
$\mathcal{B}/G$, and $p : \mathcal{B} \to M$ the canonical projection.
By the above observation, there is a canonical bijection from $M$ onto
the set of isomorphism classes of Schur representations $(W,\sigma)$
of $Q$, such that $\dim{W,\sigma} = d$.  We will call $M$ the
\emph{moduli space} of Schur representations of $Q$ with dimension
vector $d$.  Note that the action of $G$ on $\mathcal{A}$ induces a
holomorphic right action of $\bar{G}$ on $\mathcal{B}$.

The Lie algebra $\Lie{G}$ of $G$ is the direct sum Lie algebra
$\bigoplus_{a\in Q_0} \End[\C]{V_a}$, where, for each $a\in Q_0$, the
associative $\C$-algebra $\End[\C]{V_a}$ is given its usual Lie
algebra structure.  Note that $\Lie{G}$ has a canonical structure of
an associative unital $\C$-algebra, and that $G$ is the group of units
of the underlying ring of $\Lie{G}$, and is open in $\Lie{G}$.  The
Lie algebra of $H$ is the Lie subalgebra of $\Lie{G}$ consisting of
all elements of the form $ce$, as $c$ runs over $\C$.  Let
$\mathrm{Tr} : \Lie{G} \to \C$ be the $\C$-linear function defined by
$\tr{\xi} = \sum_{a\in Q_0}\tr{\xi_a}$ for all elements
$\xi = (\xi_a)_{a\in Q_0}$ of $\Lie{G}$, and let $\Lie{G}^0$ denote
its kernel.  Then, as $d\neq 0$,
$\tr{e} = \rk{d} = \sum_{a\in Q_0} d_a$ is a non-zero natural number,
and we have a decomposition $\Lie{G} = \Lie{H} \oplus \Lie{G}^0$.  As
$d$ is a non-zero element of $\N^{Q_0}$, we have $e\neq 0$, hence the
map $t\mapsto te : \C \to \Lie{H}$ is a $\C$-isomorphism.  For any
element $\xi$ of $\Lie{G}$, we define $(c(\xi),\xi^0)$ to be the
unique element of $\C \times \Lie{G}^0$, such that
$\xi = c(\xi)e+\xi^0$.  Then, $c(\xi) = \frac{\tr{\xi}}{\rk{d}}$, and
$\xi^{0} = \xi - c(\xi)e$ for all $\xi\in \Lie{G}$.

For every element $\rho$ of $\mathcal{A}$, we denote the orbit map
$g\mapsto \rho g :G \to \mathcal{A}$ by $\mu_\rho$, and by $D_\rho$
the $\C$-linear map $\htang[e]{\mu_\rho} : \Lie{G} \to \mathcal{A}$.
Thus,
\begin{equation*}
  \label{eq:235}
  D_\rho(\xi) =
  (\rho_\alpha \circ \xi_{\sr(\alpha)} -
  \xi_{\tg(\alpha)} \circ \rho_\alpha)_{\alpha\in Q_1}
\end{equation*}
for all $\xi \in \Lie{G}$.  Therefore, $\Ker{D_\rho} = \End{V,\rho}$.
In particular, $\Ker{D_\rho} = \Lie{H}$ if $\rho\in \mathcal{B}$.

It will be convenient to fix a family $h=(h_a)_{a\in Q_0}$ of
Hermitian inner products $h_a : V_a \times V_a \to \C$.  Thus, for
every point $\rho\in \mathcal{A}$, $h$ is a Hermitian metric on the
representation $(V,\rho)$ of $Q$.

For any two finite-dimensional Hermitian inner product spaces $V$ and
$W$, we have a Hermitian inner product $\pair{\cdot}{\cdot}$ on the
$\C$-vector space $\Hom[\C]{V}{W}$, which is defined by
$\pair{u}{v} = \tr{u\circ v^*}$ for all $u,v\in \Hom[\C]{V}{W}$,
where, $v^* : W \to V$ is the adjoint of $v$.  If we denote the norm
associated to this Hermitian inner product by $\norm{\cdot}$, then
$\norm{u(x)} \leq \norm{u} \norm{x}$ for all $u\in \Hom[\C]{V}{W}$ and
$x\in V$.  Also, $\norm{u^*} = \norm{u}$,
$\norm{\id{V}} = \sqrt{\dim[\C]{V}}$, and, for all finite-dimensional
Hermitian inner product spaces $V$, $W$, and $X$, and for all
$u\in \Hom[\C]{V}{W}$ and $v\in \Hom[\C]{W}{X}$, we have
$\norm{v\circ u} \leq \norm{v}\norm{u}$.

\begin{remark}
  \label{rem:9}
  Using the above facts, it is easy to verify that for every
  $u\in \Hom[\C]{V}{W}$, there exists a real number $\theta>0$, such
  that $\theta \norm{x} \leq \norm{u(x)}$ for all
  $x\in \Ker{u}^\perp$, where $X^\perp$ denotes the orthogonal
  complement of any subset $X$ of a finite-dimensional Hermitian inner
  product space.
\end{remark}

In particular, the family $h$ induces a Hermitian inner product
$\pair{\cdot}{\cdot}$ on the $\C$-vector space $\Hom[\C]{V_a}{V_b}$
for all $a,b\in Q_0$.  We give $\Lie{G}$ the Hermitian inner product
$\pair{\cdot}{\cdot}$ which is the direct sum of the Hermitian inner
products $\pair{\cdot}{\cdot}$ on $\End[\C]{V_a}$ as $a$ runs over
$Q_0$.  Note that $\norm{e} = \sqrt{\rk{d}}$ with respect to this
Hermitian inner product, and that $\Lie{H}^\perp = \Lie{G}^0$.  Let
$u: \Lie{G} \to \Lie{H}$ be the corresponding orthogonal projection.

Similarly, we give $\mathcal{A}$ the Hermitian inner product
$\pair{\cdot}{\cdot}$ which is the direct sum of the Hermitian inner
products $\pair{\cdot}{\cdot}$ on
$\Hom[\C]{V_{\sr(\alpha)}}{V_{\tg(\alpha)}}$ as $\alpha$ runs over
$Q_1$.  For every $\rho\in \mathcal{A}$, we have the adjoint
$D_\rho^* : \mathcal{A} \to \Lie{G}$ of the $\C$-linear map
$D_\rho : \Lie{G} \to \mathcal{A}$ that was defined above.

\begin{theorem}
  \label{thm:1}
  The action of $\bar{G}$ on $\mathcal{B}$ is principal.  In
  particular, there exists a unique structure of a complex premanifold
  on the moduli space $M$ of complex Schur reprsentations of $Q$ with
  dimension vector $d$, such that $p : \mathcal{B} \to M$ is a
  holomorphic submersion.  Moreover, this structure makes the map $p$
  a holomorphic principal $\bar{G}$-bundle.
\end{theorem}

\proof Let $\rho$ be an arbitrary point of $\mathcal{B}$.  Then, as
observed above, $\mathcal{B}$ is a $G$-invariant open complex
submanifold of $\mathcal{A}$, so there is an induced holomorphic right
action of $G$ on $\mathcal{B}$.  Also, $G_\rho = H$ for all
$\rho\in \mathcal{B}$.  Therefore, by Corollary~\ref{cor:3}, it
suffices to prove that for every $H$-invariant neighbourhood $V$ of
$e$ in $G$, there exists an open neighbourhood $U$ of $\rho$ in
$\mathcal{B}$, such that $P_G(U,U)\subset V$.

Let $D_\rho : \Lie{G} \to \mathcal{A}$ be the $\C$-linear map defined
earlier.  Then, as noted above, $\Ker{D_\rho} = \Lie{H}$.  Therefore,
$\Ker{D_\rho}^\perp = \Lie{H}^\perp = \Lie{G}^0$, hence, by
Remark~\ref{rem:9}, there exists a real number $\theta>0$, such that
$\theta\norm{f} \leq \norm{D_\rho(f)}$ for all $f\in \Lie{G}^0$.  Let
$q_1=\card{Q_1}$.  Then, the continuity of the norm function on
$\mathcal{A}$ implies that the set $X$ of all
$(\sigma,\tau) \in \mathcal{A} \times \mathcal{A}$, such that
$q_1(\norm{\sigma-\rho}+\norm{\tau-\rho}) < \theta$, is an open
neighbourhood of $(\rho,\rho)$ in $\mathcal{A}\times \mathcal{A}$.

Consider a point $(\sigma,\tau) \in X$, let $g\in G$, and suppose
$\tau=\sigma g$.  Then, by the above paragraph,
\begin{equation*}
  \label{eq:241}
  \theta\norm{g^0} \leq \norm{D_\rho(g^0)}.
\end{equation*}
Let $\sigma' = \sigma-\rho$ and $\tau'=\tau-\rho$.  The relation
$\tau = \sigma g$ implies that
\begin{equation*}
  \label{eq:242}
  D_\rho(g) =
  \bigl(
  g_{\tg(\alpha)}\circ \tau'_\alpha -
  \sigma'_\alpha \circ g_{\sr(\alpha)}
  \bigr)_{\alpha\in Q_1}.
\end{equation*}
Therefore, as $c(g)e\in \Ker{D_\rho}$, we have
\begin{align*}
  \norm{D_\rho(g^0)}
  &=
    \norm{D_\rho(g)}
    \leq
    \sum_{\alpha\in Q_1}
    \norm{g_{\tg(\alpha)}\circ \tau'_\alpha -
    \sigma'_\alpha \circ g_{\sr(\alpha)}} \\
  &\leq
    q_1
    \bigl(
    \abs{c(g)}\sqrt{\rk{d}} + \norm{g^0}
    \bigr)
    \bigl(
    \norm{\tau'} + \norm{\sigma'}
    \bigr).
\end{align*}
Thus,
\begin{equation*}
  \label{eq:243}
  \theta\norm{g^0} \leq
  q_1
  \bigl(
  \abs{c(g)}\sqrt{\rk{d}} + \norm{g^0}
  \bigr)
  \bigl(
  \norm{\sigma'} + \norm{\tau'}
  \bigr).
\end{equation*}

As $(\sigma,\tau)\in X$, we have
$q_1(\norm{\sigma'}+\norm{\tau'}) < \theta$, hence this implies that
\begin{equation*}
  \label{eq:245}
  \norm{g^0} \leq
  \frac{q_1
    \abs{c(g)}\sqrt{\rk{d}} (\norm{\sigma'} + \norm{\tau'})}{\theta
    -  q_1(\norm{\sigma'} + \norm{\tau'})}.
\end{equation*}
In particular, $c(g)\neq 0$; for, if $c(g)=0$, then, by the above
inequality, we get $g^0=0$, hence $g=c(g)e+g^0=0$; therefore, as
$g_a\in \Aut[\C]{V_a}$, we have $V_a=0$ for every $a\in Q_0$, a
contradiction, since $d$ is a non-zero element of $\N^{Q_0}$.  It
follows that
\begin{equation*}
  \label{eq:246}
  \left\Vert \frac{1}{c(g)}g - e \right\Vert
  =
  \left\Vert \frac{1}{c(g)}g^0 \right\Vert
  =
  \frac{1}{\abs{c(g)}} \norm{g^0}
  \leq
  \frac{q_1
    \sqrt{\rk{d}} (\norm{\sigma'} + \norm{\tau'})}{\theta
    -  q_1(\norm{\sigma'} + \norm{\tau'})}.
\end{equation*}

We have thus shown that for all $(\sigma,\tau)\in X$ and $g\in G$,
such that $\tau=\sigma g$, we have $c(g)\neq 0$, and
\begin{equation*}
  \label{eq:247}
  \left\Vert \frac{1}{c(g)}g - e \right\Vert
  \leq
  \delta(\sigma,\tau),
\end{equation*}
where $\delta : X \to [0,+\infty)$ is the function defined by
\begin{equation*}
  \label{eq:248}
  \delta(\sigma,\tau)
  =
  \frac{q_1
    \sqrt{\rk{d}} (\norm{\sigma-\rho} + \norm{\tau-\rho})}{\theta
    -  q_1(\norm{\sigma-\rho} + \norm{\tau-\rho})}.
\end{equation*}

Now, let $V$ be an $H$-invariant open neighbourhood of $e$ in $G$.
Then, as $G$ is open in $\Lie{G}$, there exists $\epsilon>0$, such
that the open ball $B(e,\epsilon)$ in $\Lie{G}$, with radius
$\epsilon$ and centre $e$, is contained in $V$.  As $X$ is an open
neighbourhood of $(\rho,\rho)$ in $\mathcal{A}\times \mathcal{A}$, and
the function $\delta$ is continuous, there exists an open
neighbourhood $U$ of $\rho$ in $\mathcal{B}$, such that
$U\times U \subset X$, and $\delta(\sigma,\tau) < \epsilon$ for all
$(\sigma,\tau) \in U\times U$.  We claim that $P_G(U,U)\subset V$.
Let $g\in P_G(U,U)$.  Then, there exists a point $(\sigma,\tau)$ of
$U\times U$, such that $\tau=\sigma g$.  As $(\sigma,\tau)\in X$, by
the above paragraph, $c(g)\neq 0$, and
$\norm{\frac{1}{c(g)}g-e} \leq \delta(\sigma,\tau) < \epsilon$, hence
$\frac{1}{c(g)}g \in B(e,\epsilon) \subset V$.  As $V$ is
$H$-invariant,
$g = \bigl(c(g)e\bigr) \bigl(\frac{1}{c(g)}g\bigr) \in HV \subset V$.
Thus, $P_G(U,U) \subset V$.  \qed

\section{The K\"ahler metric on moduli of stable representations}
\label{sec:kahler-metric-moduli}

\subsection{Moment maps}
\label{sec:moment-maps}

Let $(X,\Omega)$ be a smooth symplectic manifold.  Recall that a
smooth vector field $\xi$ on $X$ is called \emph{symplectic} if
$L_\xi(\Omega) = 0$, where $L_\xi$ is the Lie derivative with respect
to $\xi$.  The set $\vect{X,\Omega}$ of symplectic vector fields on
$X$ is a Lie subalgebra of the real Lie algebra $\vect{X}$ of smooth
vector fields on $X$.  For any smooth real function $f$ on $X$, we
denote by $H(f)$ the Hamiltonian vector field of $f$, that is, the
unique smooth vector field on $X$, such that
$\Omega(x)(H(f)(x),w) = w(f)$ for all $x\in X$ and $w\in \tang[x]{X}$,
where $\tang[x]{X}$ denotes the tangent space of $X$ at $x$; it is a
symplectic vector field.  We define the Poisson bracket of any two
smooth real functions $f$ and $g$ on $X$, by
$\{f,g\} = \Omega(H(g),H(f))$.  This makes the $\R$-vector space
$S(X)$ of smooth real functions on $X$ a Lie algebra, and the map
$f\mapsto H(f)$ is a homomorphism of real Lie algebras
$H : S(X) \to \vect{X,\Omega}$.

Let $K$ be a real Lie group.  Suppose we are given a smooth right
action of $K$ on $X$, which is symplectic, by which we mean that it
preserves the symplectic form $\Omega$ on $X$.  Thus, $\Omega$ is
$K$-invariant, that is, $\rho_g^*(\Omega) = \Omega$ for all $g\in K$,
where $\rho_g$ denotes the translation by $g$ on every right
$K$-space.  Then, for every element $\xi$ of the Lie algebra $\Lie{K}$
of $K$, the induced vector field $\xi^\sharp$ on $X$ is symplectic.
The map $\xi \mapsto \xi^\sharp$ is a homomorphism of real Lie
algebras from $\Lie{K}$ to $\vect{X,\Omega}$, which is $K$-equivariant
for the adjoint action of $K$ on $\Lie{K}$, and the canonical action
of $K$ on $\vect{X,\Omega}$.

A \emph{moment map} for the action of $K$ on $X$ is a smooth map
$\Phi : X \to \Lie{K}^*$, which is $K$-invariant for the coadjoint
action of $K$ on $\Lie{K}^*$, and has the property that
$H(\Phi^\xi)=\xi^\sharp$ for all $\xi \in \Lie{K}$, where $\Phi^\xi$
is the smooth real function $x\mapsto \Phi(x)(\xi)$ on $X$.  If $\Phi$
is a moment map, then, for every $x\in X$, the $\R$-linear map
$\tang[x]{\Phi} : \tang[x]{X} \to \Lie{K}^*$ is given by
$\tang[x]{\Phi}(v)(\xi) = \Omega(x)(\xi^\sharp(x),v)$ for all
$v\in \tang[x]{X}$ and $\xi\in \Lie{K}$.  Thus,
$\Ker{\tang[x]{\Phi}} = \Img{\tang[e]{\nu_x}}^{\perp(\Omega)}$, where
$\nu_x : K \to X$ is the orbit map of $x$,
$\tang[e]{\nu_x} : \Lie{K} \to \tang[x]{X}$ is the induced $\R$-linear
map, and $S^{\perp(\Omega)}$ denotes the set of elements of
$\tang[x]{X}$ that are $\Omega(x)$-orthogonal to any subset $S$ of
$\tang[x]{X}$.  This implies that
$\Ker{\tang[x]{\Phi}}^{\perp(\Omega)} = \Img{\tang[e]{\nu_x}}$.  Also,
$\Img{\tang[x]{\Phi}} = \Ann{\Lie{K_x}}$, where $K_x$ is the
stabiliser of $x$ in $K$, and $\Ann{M}$ denotes the annihilator in
$\Lie{K}^*$ of any subset $M$ of $\Lie{K}$.  In particular, $\Phi$ is
a submersion at $x$ if and only if the subgroup $K_x$ of $K$ is
discrete.

For later use, we state here a simple fact about moment maps for
linear actions.  Let $V$ be a finite-dimensional $\R$-vector space,
and $\Omega$ a symplectic form on $V$.  Since the tangent space of $V$
at any point is canonically isomorphic to $V$ itself, $\Omega$ defines
a smooth $2$-form on $V$.  By abuse of notation, we will denote this
smooth $2$-form also by $\Omega$.  In any linear coordinate system on
$V$, we can express $\Omega$ as a form with constant coefficients, so
$\diff \Omega = 0$.  Therefore, $(V,\Omega)$ is a symplectic manifold.

Let $K$ be a real Lie group, and suppose that we are given a smooth
linear right action of $K$ on $V$, which preserves the symplectic form
$\Omega$ on $V$.  For any element $\xi$ of $\Lie{K}$, let $\xi^\sharp$
be the vector field on $V$ defined by $\xi$; then $\xi^\sharp$ is an
$\R$-endomorphism of $V$, and
$\Omega(\xi^\sharp(x),y)+\Omega(x,\xi^\sharp(y)) = 0$ for all
$x,y\in V$.  For each element $\alpha$ of $\Lie{K}^*$, define a map
$\Phi_\alpha : V \to \Lie{K}^*$ by
\begin{equation*}
  \label{eq:341}
  \Phi_\alpha(x)(\xi) =
  \frac{1}{2}\Omega(\xi^\sharp(x),x) + \alpha(\xi)
\end{equation*}
for all $x\in V$ and $\xi\in \Lie{K}$.

\begin{lemma}
  \label{lem:14}
  The map $\alpha \mapsto \Phi_\alpha$ is a bijection from the set of
  $K$-invariant elements of $\Lie{K}^*$ onto the set of moment maps
  for the action of $K$ on $V$.
\end{lemma}

\proof It is easy to see that $\Phi_0$ is a moment map for the action
of $K$ on $V$.  Let $\alpha\in \Lie{K}^*$.  Then,
$\Phi_\alpha = \Phi_0 + \alpha$.  Therefore,
$H(\Phi_\alpha^\xi) = H(\Phi_0) = \xi^\sharp$ for all
$\xi\in \Lie{K}$; moreover, since $\Phi_0$ is $K$-equivariant,
$\Phi_\alpha$ is $K$-equivariant if and only if $\alpha$ is
$K$-invariant, that is, $\alpha(\Ad(g)\xi) = \alpha(\xi)$ for all
$g\in K$ and $\xi \in \Lie{K}$.  Thus, $\Phi_\alpha$ is a moment map
for the action of $K$ if and only if $\alpha$ is $K$-invariant.

Now, the map $\alpha\mapsto \Phi_\alpha$ is clearly injective.
Suppose $\Psi : V \to \Lie{K}^*$ is a moment map for the action of $K$
on $V$.  Then, for all $\xi\in \Lie{K}$, we have
$H(\Psi^\xi) = H(\Phi_0^\xi) = \xi^\sharp$, hence
$\tang[x]{\Psi^\xi} = \tang[x]{\Phi_0^\xi}$ for all $x\in X$; as $V$
is connected, this implies that $\Psi^\xi - \Phi_0^\xi = \alpha(\xi)$
for some $\alpha(\xi)\in \R$.  The function
$\alpha : \xi \mapsto \alpha(\xi)$ is obviously $\R$-linear, that is,
an element of $\Lie{K}^*$.  As we have seen above, it is necessarily
$K$-invariant.  Thus, the given map is surjective too.  \qed

\subsection{K\"ahler quotients}
\label{sec:kahler-quotients}

For any complex premanifold $X$ and $x\in X$, we will identify the
tangent space at $x$ of the underlying smooth manifold of $X$, with
the holomorphic tangent space $\htang[x]{X}$ of $X$ at $x$, using the
canonical $\R$-isomorphism between them.  With this identification,
for any holomorphic map $f: X \to Y$ of complex premanifolds, and
$x\in X$, the real differential of $f$ at $x$ is equal to the
$\C$-linear map $\htang[x]{f} : \htang[x]{X} \to \htang[f(x)]{Y}$,
considered as an $\R$-linear map.

Let $X$ be a complex manifold, $g$ a K\"ahler metric on $X$, and
$\Omega$ its K\"ahler form, that is, the closed real $2$-form on $X$
defined by $\Omega(x)(v,w) = -2 \Im(g(x)(v,w))$ for all $x\in X$, and
$v,w\in \tang[x]{X}$, where $\Im(t)$ denotes the imaginary part of a
complex number $t$.  Since $\Omega$ is positive, $(X,\Omega)$ is a
smooth symplectic manifold. The following Lemma follows directly from
the definition.

\begin{lemma}
  \label{lem:15}
  Let $B$ be the smooth Riemannian metric on $X$ defined by
  $B(x)(v,w) = 2\Re(g(x)(v,w))$, where $\Re(t)$ denotes the real part
  of a complex number $t$.  Then, for every point $x\in X$, and
  $\R$-subspace $W$ of $\tang[x]{X}$, we have
  $W^{\perp(\Omega)} = (\sqrt{-1}W)^{\perp(B)}$, where
  $S^{\perp(\Omega)}$ (respectively, $S^{\perp(B)}$) is the set of all
  elements of $\tang[x]{X}$ that are $\Omega(x)$-orthogonal
  (respectively, $B(x)$-orthogonal) to any subset $S$ of
  $\tang[x]{X}$.  In particular, we have an $\R$-vector space
  decomposition $\tang[x]{X} = W^{\perp(\Omega)} \oplus (\sqrt{-1}W)$.
\end{lemma}

Let $G$ be a complex Lie group, and $K$ a compact subgroup of $G$; in
particular, $K$ is a real Lie subgroup of $G$.  Suppose that we are
given a holomorphic right action of $G$ on $X$, such that the induced
action of $K$ on $X$ preserves the K\"ahler metric $g$ on $X$.  Then,
the K\"ahler form $\Omega$ on $X$ is $K$-invariant, that is, the
action of $K$ on $X$ is symplectic.  Let $\Phi : X \to \Lie{K}^*$ be a
moment map for the action of $K$ on $X$, $\zm{X}$ the closed subset
$\Phi^{-1}(0)$ of $X$, and $\zms{X} = \zm{X}G$.  Then, $\zms{X}$ is
$G$-invariant, and, since $\Phi$ is $K$-equivariant, $\zm{X}$ is a
$K$-invariant subset of $\zms{X}$.  Denote by $Y$ the quotient
topological space $X/G$, and let $p : X \to Y$ be the canonical
projection.  Let $\zms{Y} = p(\zms{X})$,
$\zms{p} : \zms{X} \to \zms{Y}$ the map induced by $p$, and
$\zm{p} = \restrict{\zms{p}}{\zm{X}} : \zm{X} \to \zms{Y}$.

\begin{proposition}
  \label{pro:29}
  Suppose that the action of $G$ on $X$ is principal, and that
  \begin{equation*}
    \label{eq:312}
    P_G(\zm{X},\zm{X}) \subset K.
  \end{equation*}
  Then:
  \begin{enumerate}
  \item \label{item:77} The set $\zm{X}$ is a closed smooth
    submanifold of $X$, $\zms{X}$ is open in $X$, $\zms{Y}$ is open in
    $Y$, the action of $K$ on $\zm{X}$ is principal, and
    $\zm{p} : \zm{X} \to \zms{Y}$ is a smooth principal $K$-bundle.
  \item \label{item:78} The action of $G$ on $\zms{X}$ is proper,
    $\zms{Y}$ is a Hausdorff open subspace of $Y$, the action of $G$
    on $\zms{X}$ is principal, and $\zms{p} : \zms{X} \to \zms{Y}$ is
    a holomorphic principal $G$-bundle.
  \item \label{item:79} There exists a unique K\"ahler metric
    $\zms{h}$ on $\zms{Y}$, such that
    $\zm{p}^*(\zms{\Theta}) = \zm{\Omega}$, where $\zms{\Theta}$ is
    the K\"ahler form of $\zms{h}$, $\zm{\Omega} = \zm{i}^*(\Omega)$,
    and $\zm{i} : \zm{X} \to X$ is the inclusion map.
  \end{enumerate}
\end{proposition}

\proof (\ref{item:77}) We will use the remarks and the notation in
Section~\ref{sec:moment-maps}.  Since the action of $G$ on $X$ is
free, we have $K_x=\set{e}$ for all $x\in X$, hence the moment map
$\Phi : X \to \Lie{K}^*$ is a submersion.  Therefore,
$\zm{X} = \Phi^{-1}(0)$ is a closed submanifold of $X$, and, for all
$x\in \zm{X}$, we have
$\tang[x]{\zm{X}} = \Ker{\tang[x]{\Phi}} =
\Img{\tang[e]{\nu_x}}^{\perp(\Omega)}$.

We will next check that $\zms{X}$ is open in $X$.  Let
$\zm{\mu} : \zm{X} \times G \to X$ be the restriction of the action
map $\mu : X \times G \to X$.  We claim that the smooth map $\zm{\mu}$
is a submersion.  For all $g\in G$, we have
$\zm{\mu}\circ (\id{\zm{X}}\times \rho_g) = \rho_g \circ \zm{\mu}$,
where $\rho_g$ denotes the translation by $g$ on any right $G$-space.
Therefore, it suffices to check that the $\R$-linear map
\begin{equation*}
  \label{eq:313}
  \tang[(x,e)]{\zm{\mu}} :
  \tang[x]{\zm{X}} \oplus \Lie{G} \to \tang[x]{X}
\end{equation*}
is surjective for every $x\in \zm{X}$.  For all
$w\in \tang[x]{\zm{X}}$ and $\xi\in \Lie{G}$, we have
\begin{equation*}
  \label{eq:314}
  \tang[(x,e)]{\zm{\mu}}(w,\xi) = \tang[(x,e)]{\mu}(w,\xi) =
  w + \xi^\sharp(x) = w + \tang[e]{\mu_x}(\xi),
\end{equation*}
where $\mu_x : G \to X$ is the orbit map of $X$.  Now, putting
$W = \Img{\tang[e]{\nu_x}}$ in Lemma~\ref{lem:15}, we get
\begin{equation*}
  \label{eq:315}
  \tang[x]{X} =
  \Img{\tang[e]{\nu_x}}^{\perp(\Omega)} \oplus
  (\sqrt{-1} \Img{\tang[e]{\nu_x}}) =
  \tang[x]{\zm{X}} \oplus (\sqrt{-1} \Img{\tang[e]{\nu_x}}).
\end{equation*}
Therefore, for each $u\in \tang[x]{X}$, there exist
$w\in \tang[x]{\zm{X}}$ and $\eta\in \Lie{K}$, such that
$u = w + \sqrt{-1} \tang[e]{\nu_x}(\eta)$.  But, since
$\nu_x : K \to X$ is the restriction of $\mu_x : G \to X$, we have
$\tang[e]{\nu_x}(\eta) = \tang[e]{\mu_x}(\eta)$.  Also, since $\mu_x$
is holomorphic, the map $\tang[e]{\mu_x} : \Lie{G} \to \tang[x]{X}$ is
$\C$-linear.  Therefore,
\begin{equation*}
  \label{eq:316}
  u =
  w + \sqrt{-1} \tang[e]{\mu_x}(\eta) =
  w + \tang[e]{\mu_x}(\sqrt{-1}\eta) =
  \tang[(x,e)]{\zm{\mu}}(w,\sqrt{-1}\eta).
\end{equation*}
This proves that $\tang[(x,e)]{\zm{\mu}}$ is surjective for all
$x\in \zm{X}$, hence $\zm{\mu}$ is a submersion.  Therefore, it is an
open map.  In particular,
$\zms{X} = \zm{X}G = \zm{\mu}(\zm{X}\times G)$ is open in $X$.

The map $p : X \to Y$ is the quotient map for a continuous action of a
topological group, and is hence an open map.  Therefore, as $\zms{X}$
is open in $X$, $\zms{Y} = p(\zms{X})$ is open in $Y$.  Moreover, as
$\zms{X}$ is $G$-invariant, we have
$\zms{X} = p^{-1}(p(\zms{X})) = p^{-1}(\zms{Y})$.  Now, by
Proposition~\ref{pro:27}, there exists a unique structure of a complex
premanifold on $Y$, such that $p$ is a holomorphic submersion;
moreover, this structure makes $p$ a holomorphic principal $G$-bundle.
Therefore, $\zms{p}$ is also a holomorphic principal $G$-bundle.

The map $\zm{p} : \zm{X} \to \zms{Y}$ is obviously smooth.  It is
surjective, because
$\zms{Y} = p(\zms{X}) = p(\zm{X}G) = p(\zm{X}) = \zm{p}(\zm{X})$.  We
will now check that it is a submersion.  Let $x\in \zm{X}$, and
$w\in \tang[p(x)]{Y}$.  Then, since $p : X \to Y$ is a holomorphic
submersion, there exists $v\in \tang[x]{X}$, such that
$\tang[x]{p}(v) = w$.  As
$\tang[x]{X} = \tang[x]{\zm{X}} \oplus (\sqrt{-1}
\Img{\tang[e]{\nu_x}})$, there exist $\zm{v} \in \tang[x]{\zm{X}}$ and
$\xi \in \Lie{K}$, such that $v=\zm{v}+\sqrt{-1}\tang[e]{\nu_x}(\xi)$.
Now, $\tang[e]{\nu_x}(\xi) = \tang[e]{\mu_x}(\xi)$ belongs to the
$\C$-subspace
$\Img{\tang[e]{\mu_x}} = \tang[x]{xG} = \Ker{\tang[x]{p}}$ of
$\tang[x]{X}$, hence $\tang[x]{p}(\sqrt{-1}\tang[e]{\nu_x}(\xi)) = 0$.
Therefore, $w = \tang[x]{p}(\zm{v}) = \tang[x]{\zm{p}}(\zm{v})$.  This
proves that $\zm{p}$ is a submersion.  The condition
$P_G(\zm{X},\zm{X}) \subset K$, and the $K$-invariance of $\zm{p}$,
imply that $\zm{p}^{-1}(\zm{p}(x)) = xK$ for all $x\in \zm{X}$.
Lastly, if $R$ is the graph of the action of $G$ on $X$, $\zm{R}$ that
of the action of $K$ on $\zm{X}$, and $\phi : R \to G$ and
$\zm{\phi} : \zm{R} \to K$ the translation maps, then
$\zm{R} \subset R$, and $\zm{\phi}$ is induced by $\phi$.  As the
action of $G$ on $X$ is principal, $\phi$ is continuous, hence so is
$\zm{\phi}$, so the action of $K$ on $\zm{X}$ is also principal.  Now,
by Remark~\ref{rem:8}, $\zm{p}$ is a smooth principal $K$-bundle.

(\ref{item:78}) As $\zm{p} : \zm{X} \to \zms{Y}$ is a smooth principal
$G$-bundle, it is an open map.  Therefore,
$\zm{p} \times \zm{p} : \zm{X} \times \zm{X} \to \zms{Y} \times
\zms{Y}$ is also an open map.  Since it is also a continuous
surjection, it is a quotient map.  Now,
$\zm{R} = (\zm{p} \times \zm{p})^{-1}(\zms{\Delta})$, where $\zm{R}$
is the graph of the action of $K$ on $\zm{X}$, and $\zms{\Delta}$ is
the diagonal of $\zms{Y}$.  Since $K$ is compact and $\zm{X}$ is
Hausdorff, the action of $K$ on $\zm{X}$ is proper, hence $\zm{R}$ is
closed in $\zm{X} \times \zm{X}$.  Therefore, $\zms{\Delta}$ is closed
in $\zms{Y} \times \zms{Y}$, so $\zms{Y}$ is Hausdorff.  The graph
$\zms{R}$ of the action of $G$ on $\zms{X}$ equals
$(\zms{p}\times \zms{p})^{-1}(\zms{\Delta})$, and is hence closed in
$\zms{X}\times \zms{X}$.  Moreover, $\zms{R}$ is contained in the
graph $R$ of the action of $G$ on $X$, and the translation map
$\zms{\phi} : \zms{R} \to G$ is the restriction of the translation map
$\phi : R \to G$.  As the action of $G$ on $X$ is principal, this
implies that the action of $G$ on $\zms{X}$ is also principal.  Let
$\zms{\sigma} : \zms{X}\times G \to \zms{X}\times \zms{X}$ be the map
$(x,g)\mapsto (x,xg)$, and let
$\zms{\tau} : \zms{X}\times G \to \zms{R}$ be the map induced by
$\zms{\sigma}$.  Then, by Remark (\ref{rem:7}), $\zms{\tau}$ is a
homeomorphism.  Since $\zms{R}$ is closed in $\zms{X}\times \zms{X}$, it
follows that the map $\zms{\sigma}$ is proper.  In other words, the
action of $G$ on $\zms{X}$ is proper.  It has been proved above that
$\zms{Y}$ is open in $Y$, and $\zms{p}$ is a holomorphic principal
$G$-bundle.

(\ref{item:79}) By hypothesis, the K\"ahler metric $g$ on $X$ is
$K$-invariant, hence its K\"ahler form $\Omega$ is $K$-invariant.
Therefore, its restriction $\zm{\Omega}$ to the $K$-invariant smooth
submanifold $\zm{X}$ of $X$ is also $K$-invariant.  Let $x\in \zm{X}$,
$v\in \tang[x]{\zm{X}}$, and $\xi\in \Lie{K}$.  Then, since
$\tang[x]{\zm{X}} = \Img{\tang[e]{\nu_x}}^{\perp(\Omega)}$, we have
$\zm{\Omega}(x)(v,\tang[e]{\nu_x}(\xi)) =
\Omega(x)(v,\tang[e]{\nu_x}(\xi)) = 0$.  Therefore, $\Omega(x)(v,w)=0$
if either $v$ or $w$ is a vertical tangent vector at $x$ for the
principal $K$-bundle $\zm{p} : \zm{X} \to \zms{Y}$.  It follows that
there exists a unique smooth $2$-form $\zms{\Theta}$ on $\zms{Y}$,
such that $\zm{p}^*(\zms{\Theta}) = \zm{\Omega}$.  As $\Omega$ is
closed, so is $\zm{\Omega}$, and hence so is $\zms{\Theta}$.

We claim that $\zms{\Theta}$ is positive.  Let $y\in \zms{Y}$, and
$w,w' \in \tang[y]{Y}$.  Let $x\in \zm{p}^{-1}(y)$.  Then, there
exists $v,v'\in \tang[x]{\zm{X}}$, such that $w = \tang[x]{\zm{p}}(v)$
and $w' = \tang[x]{\zm{p}}(v')$.  Since
$\tang[x]{X} = \tang[x]{\zm{X}} \oplus (\sqrt{-1}
\Img{\tang[e]{\nu_x}})$, there exist $a,a'\in \tang[x]{\zm{X}}$ and
$\xi,\xi'\in \Lie{K}$, such that
$\sqrt{-1}v = a + \sqrt{-1}\tang[e]{\nu_x}(\xi)$ and
$\sqrt{-1}v' = a' + \sqrt{-1}\tang[e]{\nu_x}(\xi')$.  As $\tang[x]{p}$
is $\C$-linear,
\begin{align*}
  \sqrt{-1}w
  &= \sqrt{-1}\tang[x]{\zm{p}}(v) =
    \sqrt{-1}\tang[x]{p}(v) =
    \tang[x]{p}(\sqrt{-1}v) \\
  &=
    \tang[x]{p}(a + \sqrt{-1}\tang[e]{\nu_x}(\xi)) =
    \tang[x]{p}(a) + \sqrt{-1}\tang[x]{p}(\tang[e]{\nu_x}(\xi)) \\
  &=
    \tang[x]{\zm{p}}(a) +
    \sqrt{-1}\tang[x]{p}(\tang[e]{\mu_x}(\xi)) =
    \tang[x]{\zm{p}}(a).
\end{align*}
Similarly, $\sqrt{-1}w' = \tang[x]{\zm{p}}(a')$.  Therefore, as
$\Omega(x)$ vanishes on vertical tangent vectors for $\zm{p}$,
\begin{align*}
  \zms{\Theta}(y)(\sqrt{-1}w,\sqrt{-1}w')
  &=
    \zms{\Theta}(y)(\tang[x]{\zm{p}}(a),\tang[x]{\zm{p}}(a')) =
    \zm{\Omega}(x)(a,a') =
    \Omega(x)(a,a') \\
  &=
    \Omega(x)(\sqrt{-1}(v-\tang[e]{\nu_x}(\xi)),
    \sqrt{-1}(v'-\tang[e]{\nu_x}(\xi'))) \\
  &=
    \Omega(x)(v-\tang[e]{\nu_x}(\xi),
    v'-\tang[e]{\nu_x}(\xi')) \\
  &=
    \zm{\Omega}(x)(v-\tang[e]{\nu_x}(\xi),
    v'-\tang[e]{\nu_x}(\xi')) \\
  &=
    \zm{\Omega}(x)(v,v') = \zms{\Theta}(y)(w,w').
\end{align*}
Similarly, it can be checked that
\begin{align*}
  \zms{\Theta}(y)(w,\sqrt{-1}w)
  = \Omega(x)(v-\tang[e]{\nu_x}(\xi),
  \sqrt{-1}(v-\tang[e]{\nu_x}(\xi))).
\end{align*}
Now, if $w = \tang[x]{\zm{p}}(v)$ is non-zero, then
$v\neq \tang[e]{\nu_x}(\xi)$, hence, as $\Omega(x)$ is positive,
$\Omega(x)(v-\tang[e]{\nu_x}(\xi), \sqrt{-1}(v-\tang[e]{\nu_x}(\xi)))
> 0$.  Therefore, $\zms{\Theta}(y)(w,\sqrt{-1}w) > 0$.  It follows
that $\zms{\Theta}$ is positive.  Thus, the rule
\begin{equation*}
  \label{eq:317}
  \zms{h}(y)(w,w') =
  \frac{1}{2} \bigl(
  \zms{\Theta}(y)(w,\sqrt{-1}w') - \sqrt{-1} \zms{\Theta}(y)(w,w')
  \bigr),
\end{equation*}
for all $y\in \zms{Y}$ and $w,w'\in \tang[y]{Y}$, defines a K\"ahler
metric on $\zms{Y}$, whose K\"ahler form equals $\zms{\Theta}$.  If
$h$ is another K\"ahler metric on $\zms{Y}$, whose K\"ahler form
$\Theta$ satisfies the condition $\zm{p}^*(\Theta) = \zm{\Omega}$,
then, since $\zms{\Theta}$ is the unique smooth $2$-form on $\zms{Y}$
such that $\zm{p}^*(\zms{\Theta}) = \zm{\Omega}$, we get
$\Theta = \zms{\Theta}$.  Therefore,
\begin{equation*}
  \label{eq:318}
  h(y)(w,w') =
  \frac{1}{2} \bigl(
  \Theta(y)(w,\sqrt{-1}w') - \sqrt{-1} \Theta(y)(w,w') \bigr) =
  \zms{h}(y)(w,w')
\end{equation*}
for all $y\in \zms{Y}$ and $w,w'\in \tang[y]{Y}$.  This establishes
the uniqueness of $\zms{h}$.  \qed


In addition to the notations used above, let $H$ be a normal complex
Lie subgroup of $G$, $\bar{G}$ the complex Lie group $H\backslash G$,
and $\pi: G \to \bar{G}$ the canonical projection.  Let $\bar{K}$ be
the compact subgroup $\pi(K)$ of $\bar{G}$, and
$\pi_K : K \to \bar{K}$ the homomorphism of real Lie groups induced by
$\pi$.  The subset $H\cap K$ of $G$ is a real Lie subgroup of $G$, and
$\Lie{H\cap K}$ equals the real Lie subalgebra $\Lie{H}\cap \Lie{K}$
of $\Lie{G}$.  The map $\tang[e]{\pi} : \Lie{G} \to \Lie{\bar{G}}$ is
a surjective homomorphism of complex Lie algebras with kernel
$\Lie{H}$, and $\tang[e]{\pi_K} : \Lie{K} \to \Lie{\bar{K}}$ is a
surjective homomorphism of real Lie algebras with kernel
$\Lie{H\cap K}$.

\begin{corollary}
  \label{cor:4}
  Suppose that $G_x=H$ for all $x\in X$, that the induced action of
  $\bar{G}$ on $X$ is principal, and that
  \begin{equation*}
    \label{eq:319}
    \Phi(X)\subset \Ann{\Lie{H\cap K}}, \quad
    P_G(\zm{X},\zm{X})\subset HK.
  \end{equation*}
  Then:
  \begin{enumerate}
  \item \label{item:80} The set $\zm{X}$ is a closed smooth
    submanifold of $X$, $\zms{X}$ is open in $X$, $\zms{Y}$ is open in
    $Y$, the action of $\bar{K}$ on $\zm{X}$ is principal, and
    $\zm{p} : \zm{X} \to \zms{Y}$ is a smooth principal
    $\bar{K}$-bundle.
  \item \label{item:81} The action of $\bar{G}$ on $\zms{X}$ is
    proper, $\zms{Y}$ is a Hausdorff open subspace of $Y$, the action
    of $\bar{G}$ on $\zms{X}$ is principal, and
    $\zms{p} : \zms{X} \to \zms{Y}$ is a holomorphic principal
    $\bar{G}$-bundle.
  \item \label{item:82} There exists a unique K\"ahler metric
    $\zms{h}$ on $\zms{Y}$, such that
    $\zm{p}^*(\zms{\Theta}) = \zm{\Omega}$, where $\zms{\Theta}$ is
    the K\"ahler form of $\zms{h}$, $\zm{\Omega} = \zm{i}^*(\Omega)$,
    and $\zm{i} : \zm{X} \to X$ is the inclusion map.
  \end{enumerate}
\end{corollary}

\proof The action of $\bar{K}$ on $X$ induced by that of $\bar{G}$ on
$X$ preserves the K\"ahler metric $g$ on $X$.  Since
$\Phi(X) \subset \Ann{\Lie{H\cap K}}$ and
$\Ker{\tang[e]{\pi_K}} = \Lie{H\cap K}$, there exists a unique map
$\bar{\Phi} : X \to \Lie{\bar{K}}^*$, such that
$\Phi(x) = \bar{\Phi}(x) \circ \tang[e]{\pi_K}$ for all $x\in X$.  It
is easy to see that $\bar{\Phi}$ is a moment map for the action of
$\bar{K}$ on $X$.  Moreover,
$\bar{\Phi}^{-1}(0) = \Phi^{-1}(0) = \zm{X}$,
$\bar{\Phi}^{-1}(0)\bar{G} = \zm{X}G = \zms{X}$, and
$P_{\bar{G}}(\zm{X},\zm{X}) \subset \pi(P_G(\zm{X},\zm{X})) \subset
\pi(HK) \subset \pi(K) = \bar{K}$.  It is obvious that
$X/\bar{G} = X/G = Y$, and the canonical projection from $X$ to
$X/\bar{G}$ equals $p$.  Therefore, the Corollary follows from
Proposition~\ref{pro:29}.  \qed

\subsection{The K\"ahler metric on the moduli of stable
  representations}
\label{sec:kahler-metric-moduli-1}

We will follow the notation of Section~\ref{sec:compl-prem-schur}.
Thus, $Q$ is a non-empty finite quiver, $d=(d_a)_{a\in Q_0}$ a
non-zero element of $\N^{Q_0}$, and $V=(V_a)_{a\in Q_0}$ a family of
$\C$-vector spaces, such that $\dim[\C]{V_a} = d_a$ for all
$a\in Q_0$.  Fix a family $h=(h_a)_{a\in Q_0}$ of Hermitian inner
products $h_a : V_a \times V_a \to \C$.  In additon, we also fix now a
rational weight $\theta \in \Q^{Q_0}$ of $Q$.

Denote by $\mathcal{A}$ the finite-dimensional $\C$-vector space
$\bigoplus_{\alpha\in Q_1}\Hom[\C]{V_{\sr(\alpha)}}{V_{\tg(\alpha)}}$.
For any subset $X$ of $\mathcal{A}$, let $\schur{X}$ (respectively,
$\stab{X}$) denote the set of all points $\rho$ in $X$, such that the
representation $(V,\rho)$ of $Q$ is Schur (respectively,
$\theta$-stable).  Also, denote by $\eh{X}$ (respectively, $\irr{X}$)
the set of all $\rho\in X$, such that the Hermitian metric $h$ on
$(V,\rho)$ is Einstein-Hermitian with respect to $\theta$
(respectively, irreducible).

Recall that $G$ is the complex Lie group
$\prod_{a\in Q_0}\Aut[\C]{V_a}$, with its canonical holomorphic linear
right action on $\mathcal{A}$.  Denote by $H$ the central complex Lie
subgroup $\units{\C}e$ of $G$, $\bar{G}$ the complex Lie group
$H\backslash G$, and $\pi : G \to \bar{G}$ the canonical projection.
Let $K$ denote the compact subgroup $\prod_{a\in Q_0} \Aut{V_a,h_a}$,
where, for each $a\in Q_0$, $\Aut{V_a,h_a}$ is the subgroup of
$\Aut[\C]{V_a}$ consisting of $\C$-automorphisms of $V_a$ which
preserve the Hermitian inner product $h_a$ on $V_a$.  Let $\bar{K}$ be
the compact subgroup $\pi(K)$ of $\bar{G}$, and
$\pi_K : K \to \bar{K}$ the homomorphism of real Lie groups induced by
$\pi$.

Let $\mathcal{B} = \schur{\mathcal{A}}$.  Then, by
Proposition~\ref{pro:2}(\ref{item:24}),
$\stab{\mathcal{B}} = \stab{\mathcal{A}}$.  On the other hand, by
Proposition~\ref{pro:16},
$\eh{\mathcal{B}} = \eh{\mathcal{A}} \cap \irr{\mathcal{A}} =
\eh{\mathcal{A}} \cap \stab{\mathcal{A}} = \eh{\mathcal{A}} \cap
\stab{\mathcal{B}}$.  As noted in Section~\ref{sec:compl-prem-schur},
$\mathcal{B}$ is a $G$-invariant open complex submanifold of
$\mathcal{A}$, and, by Proposition~\ref{pro:2}(\ref{item:25}),
$G_\rho = H$ for all $\rho\in \mathcal{B}$.  Let $M$ denote the moduli
space $\mathcal{B}/G$ of Schur representations of $Q$ with dimension
vector $d$, and $p : \mathcal{B} \to M$ the canonical projection.  It
was proved in Theorem~\ref{thm:1} that the action of $\bar{G}$ on
$\mathcal{B}$ is principal, that there exists a unique structure of a
complex premanifold on $M$ such that $p$ is a holomorphic submersion,
and that this structure in fact makes $p$ a holomorphic principal
$\bar{G}$-bundle.  Let $\stab{M} = p(\stab{\mathcal{B}})$,
$\stab{p} : \stab{\mathcal{B}} \to \stab{M}$ the map induced by $p$,
and
$\eh{p} = \restrict{\stab{p}}{\eh{\mathcal{B}}} : \eh{\mathcal{B}} \to
\stab{M}$.  Recall that for any two subsets $A$ and $B$ of
$\mathcal{A}$, $P_G(A,B)$ denotes the set of all $g\in G$, such that
$Ag\cap B \neq \emptyset$.

\begin{lemma}
  \label{lem:16}
  We have $\eh{\mathcal{B}}G = \stab{\mathcal{B}}$ and
  $P_G(\eh{\mathcal{B}},\eh{\mathcal{B}}) \subset HK$.
\end{lemma}

\proof It is obvious that the subset $\stab{\mathcal{B}}$ of
$\mathcal{B}$ is $G$-invariant.  By the above paragraph,
$\eh{\mathcal{B}} \subset \stab{\mathcal{B}}$.  Therefore,
$\eh{\mathcal{B}}G \subset \stab{\mathcal{B}}$.  Conversely, if
$\sigma\in \stab{\mathcal{B}}$, then, by Proposition~\ref{pro:14},
$(V,\sigma)$ has an Einstein-Hermitian metric $k$ with respect to
$\theta$.  For each $a\in Q_0$, $h_a$ and $k_a$ are two Hermitian
inner products on $V_a$, hence there exists a $\C$-automorphism $g_a$
of $V_a$, such that $h_a(g_a(x),g_a(y)) = k_a(x,y)$ for all
$x,y\in V_a$.  We thus get an element $g = (g_a)_{a\in Q_0}$ of $G$.
Let $\rho = \sigma g^{-1}$.  Then, for all $\alpha\in Q_1$, we have
$\rho_\alpha^{*(h)} = g_{\sr(\alpha)} \circ \sigma_\alpha^{*(k)} \circ
g_{\tg(\alpha)}^{-1}$, where $\rho_\alpha^{*(h)}$ is the adjoint of
$\rho_\alpha$ with respect to $h_{\sr(\alpha)}$ and $h_{\tg(\alpha)}$,
and $\rho_\alpha^{*(k)}$ the adjoint of $\rho_\alpha$ with respect to
$k_{\sr(\alpha)}$ and $k_{\tg(\alpha)}$.  Therefore, for every
$a\in Q_0$, we get
\begin{equation*}
  \label{eq:320}
  K_\theta(V,\rho,h)_a =
  g_a \circ K_\theta(V,\sigma,k) \circ g_a^{-1} =
  \mu_\theta(d)\id{V_a}.
\end{equation*}
Thus, the Hermitian metric $h$ on $(V,\rho)$ is Einstein-Hermitian, so
$\rho \in \eh{\mathcal{B}}$, and $\sigma = \rho g$ belongs to
$\eh{\mathcal{B}} G$.  This proves that
$\eh{\mathcal{B}}G = \stab{\mathcal{B}}$.

Next, let $g\in P_G(\eh{\mathcal{B}},\eh{\mathcal{B}})$.  Then, there
exist $\rho,\sigma\in \eh{\mathcal{B}}$, such that $\sigma = \rho g$.
Then, $g$ is an isomorphism from $(V,\sigma)$ to $(V,\rho)$.  For
every $a\in Q_0$, define a Hermitian inner product $k_a$ on $V_a$ by
$k_a(x,y) = h_a(g_a(x),g_a(y))$ for all $x,y\in V_a$.  Then, as
observed above, since $\sigma\in \eh{\mathcal{B}}$, we have
\begin{equation*}
  \label{eq:321}
  K_\theta(V,\rho,k)_a =
  g_a^{-1} \circ K_\theta(V,\sigma,h) \circ g_a =
  \mu_\theta(d)\id{V_a}
\end{equation*}
for all $a\in Q_0$, hence the Hermitian metric $k = (k_a)_{a\in Q_0}$
on $(V,\rho)$ is Einstein-Hermitian with respect to $\theta$.  As
$\rho\in \eh{\mathcal{B}}$, the Hermitian metric $h$ on $(V,\rho)$ is
also Einstein-Hermitian with respect to $\theta$.  Therefore, by
Proposition~\ref{pro:14}, there exists an automorphism $f$ of
$(V,\rho)$, such that $k_a(x,y) = h_a(f_a(x),f_a(y))$ for all
$a\in Q_0$, and $x,y\in V_a$.  Now, since $\rho\in \mathcal{B}$,
$f=ce$ for some $c\in \C$.  As the dimension vector $d$ is non-zero,
we have $c\neq 0$, and
\begin{equation*}
  \label{eq:322}
  h_a \Bigl( \frac{1}{c}g_a(x),\frac{1}{c}g_a(y) \Bigr) =
  k_a \Bigl( \frac{1}{c}x,\frac{1}{c}y \Bigr) =
  h_a \Bigl( \frac{1}{c}f_a(x),\frac{1}{c}f_a(y) \Bigr) =
  h_a(x,y)
\end{equation*}
for all $a\in Q_0$, and $x,y\in V_a$.  Therefore,
$\frac{1}{c}g_a\in \Aut{V_a,h_a}$ for all $a\in Q_0$, so
$\frac{1}{c}g\in K$.  Thus, $g=(ce)\bigl(\frac{1}{c}g\bigr)$ belongs
to $HK$.  It follows that
$P_G(\eh{\mathcal{B}},\eh{\mathcal{B}}) \subset HK$.  \qed

The family $h$ induces a Hermitian inner product $\pair{\cdot}{\cdot}$
on the $\C$-vector space $\mathcal{A}$, and the complex Lie algebra
$\Lie{G}$.  For every point $\rho$ of $\mathcal{A}$, the $\C$-vector
space $\tang[\rho]{\mathcal{A}}$ is canonically isomorphic to
$\mathcal{A}$.  Therefore, the Hermitian inner product
$\pair{\cdot}{\cdot}$ on $\mathcal{A}$ defines a Hermitian metric $g$
on the complex manifold $\mathcal{A}$, namely
$g(\rho)(\sigma,\tau) = \pair{\sigma}{\tau}$ for all
$\rho,\sigma,\tau\in \mathcal{A}$.  The fundamental $2$-form $\Omega$
of $g$ is given by
$\Omega(\rho)(\sigma,\tau) = -2\Im(\pair{\sigma}{\tau})$ for all
$\rho,\sigma,\tau\in \mathcal{A}$.  Since $\Omega(\rho)$ is
independent of $\rho$, we have $\diff \Omega = 0$, hence the Hermitian
metric $g$ on $\mathcal{A}$ is K\"ahler.  The action of $K$ on
$\mathcal{A}$ induced by that of $G$ preserves the Hermitian inner
product $\pair{\cdot}{\cdot}$, and hence the K\"ahler metric $g$ on
$\mathcal{A}$.  Similarly, the action of $K$ on $\Lie{G}$ induced by
that of $G$ preserves the Hermitian inner product on $\Lie{G}$.

For each $a\in Q_0$, let $\End{V_a,h_a}$ denote the real Lie
subalgebra of $\End[\C]{V_a}$ consisting of $\C$-endomorphisms $u$ of
$V_a$ that are skew-Hermitian with respect to $h_a$, that is,
\begin{equation*}
  \label{eq:324}
  h_a(u(x),y) + h_a(x,u(y)) = 0
\end{equation*}
for all $x,y\in V_a$.  Then, $\Lie{K}$ equals the real Lie subalgebra
$\bigoplus_{a\in Q_0} \End{V_a,h_a}$ of $\Lie{G}$.  The Hermitian
inner product $\pair{\cdot}{\cdot}$ on $\Lie{G}$ restricts to a real
inner product on $\Lie{K}$, which is given by
$\pair{\xi}{\eta} = -\sum_{a\in Q_0} \tr{\xi_a \circ \eta_a}$ for all
$\xi,\eta \in \Lie{K}$.

For any point $\rho$ in $\mathcal{A}$, let
$\nu_\rho : K \to \mathcal{A}$ be the orbit map of $\rho$, and denote
by $D_\rho$ the $\R$-linear map
$\tang[e]{\nu_\rho} : \Lie{K} \to \mathcal{A}$.  Then, as in
Section~\ref{sec:compl-prem-schur}, we have
\begin{equation*}
  \label{eq:325}
  D_\rho(\xi) =
  (\rho_\alpha \circ \xi_{\sr(\alpha)} -
  \xi_{\tg(\alpha)} \circ \rho_\alpha)_{\alpha\in Q_1}.
\end{equation*}
For every element $\xi$ of $\Lie{K}$, the vector field $\xi^\sharp$ on
$\mathcal{A}$ induced by $\xi$ is the $\C$-endomorphism of
$\mathcal{A}$ given by
$\xi^\sharp(\rho) = \tang[e]{\nu_\rho}(\xi) = D_\rho(\xi)$ for all
$\rho \in \mathcal{A}$.

Recall the notation
\begin{equation*}
  \label{eq:333}
  \deg_\theta(d) = \sum_{a\in Q_0} \theta_a d_a, \quad
  \rk{d} = \sum_{a\in Q_0} d_a, \quad
  \mu_\theta(d) = \frac{\deg_\theta(d)}{\rk{d}},
\end{equation*}
where $\theta \in \R^{Q_0}$ is the rational weight of $Q$ that we have
fixed.  If $a,b\in Q_0$, and $f \in \Hom[\C]{V_a}{V_b}$, let
$f^*\in \Hom[\C]{V_b}{V_a}$ be the adjoint of $f$ with respect to the
Hermitian inner products $h_a$ and $h_b$ on $V_a$ and $V_b$,
respectively.  For every point $\rho$ of $\mathcal{A}$, and
$a\in Q_0$, define an element $L_\theta(\rho)_a$ of $\End{V_a,h_a}$ by
\begin{equation*}
  \label{eq:326}
  L_\theta(\rho)_a =
  \sqrt{-1} \Bigl(
  (\theta_a - \mu_\theta(d)) \id{V_a} +
  \sum_{\alpha\in \tg^{-1}(a)} \rho_\alpha \circ \rho_\alpha^* -
  \sum_{\alpha\in \sr^{-1}(a)}\rho_\alpha^* \circ \rho_\alpha
  \Bigr),
\end{equation*}
and let $L_\theta(\rho)$ be the element
$(L_\theta(\rho)_a)_{a\in Q_0}$ of $\Lie{K}$.  Define a map
$\Phi_\theta : \mathcal{A} \to \Lie{K}^*$ by
\begin{equation*}
  \label{eq:335}
  \Phi_\theta(\rho)(\xi) =
  \pair{\xi}{L_\theta(\rho)}
\end{equation*}
for all $\rho\in \mathcal{A}$ and $\xi\in \Lie{K}$, where
$\pair{\cdot}{\cdot}$ is the real inner product on $\Lie{K}$.

\begin{lemma}
  \label{lem:17}
  Let $\eta$ denote the element
  $\bigl(\sqrt{-1} (\theta_a - \mu_\theta(d)) \id{V_a}\bigr)_{a\in
    Q_0}$ of $\Lie{K}$, and $\alpha$ the element of $\Lie{K}^*$, which
  is defined by $\alpha(\xi) = \pair{\xi}{\eta}$ for all
  $\xi\in \Lie{K}$.  Then,
  \begin{equation*}
    \label{eq:343}
    \Phi_\theta(\rho)(\xi) =
    \frac{1}{2}\Omega(\xi^\sharp(\rho),\rho) +
    \alpha(\xi)
  \end{equation*}
  for all $\rho\in \mathcal{A}$ and $\xi\in \Lie{K}$.  In particular,
  $\Phi_\theta$ is a moment map for the action of $K$ on
  $\mathcal{A}$.
\end{lemma}

\proof Let $\rho\in \mathcal{A}$ and $\xi\in \Lie{K}$.  For every
$a\in Q_0$, define an element $A(\rho)_a$ of $\End{V_a,h_a}$, by
\begin{equation*}
  \label{eq:334}
  A(\rho)_a =
  \sqrt{-1} \Bigl(
  \sum_{\alpha\in \tg^{-1}(a)} \rho_\alpha \circ \rho_\alpha^* -
  \sum_{\alpha\in \sr^{-1}(a)}\rho_\alpha^* \circ \rho_\alpha
  \Bigr),
\end{equation*}
and let $A(\rho)$ denote the element $(A(\rho)_a)_{a\in Q_0}$ of
$\Lie{K}$.  Then,
\begin{equation*}
  \label{eq:357}
  L_\theta(\rho) = A(\rho) + \eta, \quad
  \Phi_\theta(\rho)(\xi) =
  \pair{\xi}{A(\rho)} + \alpha(\xi).
\end{equation*}
We claim that
\begin{equation*}
  \label{eq:363}
  \pair{\xi}{A(\rho)} =
  \frac{1}{2} \Omega(\xi^\sharp(\rho),\rho).
\end{equation*}
By the definition of $\Omega$,
\begin{equation*}
  \label{eq:328}
  \frac{1}{2} \Omega(\xi^\sharp(\rho),\rho) =
  -\Im(\pair{\xi^\sharp(\rho)}{\rho}) =
  \frac{\sqrt{-1}}{2} (\pair{\xi^\sharp(\rho)}{\rho} -
  \pair{\rho}{\xi^\sharp(\rho)}).
\end{equation*}
Since $K$ preserves the Hermitian inner product on $\mathcal{A}$, the
$\C$-endomorphism $\xi^\sharp$ of $\mathcal{A}$ is skew-Hermitian,
that is,
\begin{equation*}
  \label{eq:329}
  \pair{\xi^\sharp(\rho)}{\rho} + \pair{\rho}{\xi^\sharp(\rho)} = 0.
\end{equation*}
Therefore,
\begin{equation*}
  \label{eq:330}
  \frac{1}{2} \Omega(\xi^\sharp(\rho),\rho) =
  \sqrt{-1} \pair{\xi^\sharp(\rho)}{\rho} =
  \sqrt{-1} \pair{D_\rho(\xi)}{\rho}.
\end{equation*}
It is easy to see that
\begin{equation*}
  \label{eq:331}
  \pair{D_\rho(\xi)}{\rho} =
  \sqrt{-1} \sum_{a\in Q_0} \tr{\xi_a \circ A(\rho)_a}.
\end{equation*}
Thus,
\begin{equation*}
  \label{eq:332}
  \frac{1}{2} \Omega(\xi^\sharp(\rho),\rho) =
  - \sum_{a\in Q_0} \tr{\xi_a \circ A(\rho)_a} =
  \pair{\xi}{A(\rho)},
\end{equation*}
which proves the above claim, and gives the relation
\begin{equation*}
  \label{eq:364}
  \Phi_\theta(\rho)(\xi) =
  \frac{1}{2} \Omega(\xi^\sharp(\rho),\rho) +
  \alpha(\xi)
\end{equation*}
for all $\rho\in \mathcal{A}$ and $\xi\in \Lie{K}$.  Now, the inner
product on $\Lie{K}$ is $K$-invariant, and for all $g\in K$ and
$\xi\in \Lie{K}$, we have
$\Ad(g)^{-1}\eta = (g_a^{-1}\circ \eta_a \circ g_a)_{a\in Q_0} =
\eta$, hence
\begin{equation*}
  \label{eq:362}
  \alpha(\Ad(g)\xi) = \pair{\Ad(g)\xi}{\eta} =
  \pair{\xi}{\Ad(g)^{-1}\eta} =
  \pair{\xi}{\eta} =
  \alpha(\xi).
\end{equation*}
Therefore, the element $\alpha$ of $\Lie{K}^*$ is $K$-invariant.  It
follows from Lemma~\ref{lem:14} that $\Phi_\theta$ is a moment map for
the action of $K$ on $\mathcal{A}$.  \qed

\begin{lemma}
  \label{lem:18}
  We have $\Phi_\theta(\mathcal{A}) \subset \Ann{\Lie{H\cap K}}$, and
  $\Phi_\theta^{-1}(0) = \eh{\mathcal{A}}$.
\end{lemma}

\proof Let $\rho\in \mathcal{A}$, and
$\xi \in \Lie{H\cap K} = \Lie{H}\cap \Lie{K}$.  Then, there exists a
real number $c$, such that $\xi = \sqrt{-1} c e$, where
$e=(\id{V_a})_{a\in Q_0}$ is the identity element of
$G\subset \Lie{G}$.  Therefore,
\begin{equation*}
  \label{eq:338}
  \Phi_\theta(\rho)(\xi) =
  \pair{\xi}{L_\theta(\rho)} =
  - \sum_{a\in Q_0} \tr{\xi_a \circ L_\theta(\rho)_a} =
  - \sqrt{-1} c \sum_{a\in Q_0} \tr{L_\theta(\rho)_a}.
\end{equation*}
But, with $A(\rho)$ as in the proof of Lemma~\ref{lem:17}, we have
\begin{equation*}
  \label{eq:339}
  \sum_{a\in Q_0} \tr{L_\theta(\rho)_a} =
  \sum_{a\in Q_0} \bigl(
  \tr{A(\rho)_a} +
  \sqrt{-1}(\theta_a-\mu_\theta(d))d_a \bigr) =
  \sum_{a\in Q_0} \tr{A(\rho)_a} = 0.
\end{equation*}
Therefore, $\Phi_\theta(\rho)(\xi) = 0$, hence
$\Phi_\theta(\mathcal{A}) \subset \Ann{\Lie{H\cap K}}$.

Lastly, in the notation of Section~\ref{sec:einst-herm-metr}, we have
$L_\theta(\rho) = \sqrt{-1}(K_\theta(V,\rho) - \mu_\theta(d)e)$ for
all $\rho\in \mathcal{A}$.  As $L_\theta(\rho) \in \Lie{K}$, and
$\pair{\cdot}{\cdot}$ is an inner product on $\Lie{K}$, we have
\begin{equation*}
  \label{eq:340}
  \Phi_\theta(\rho) = 0 \Liff L_\theta(\rho) = 0 \Liff
  K_\theta(V,\rho) = \mu_\theta(d) e.
\end{equation*}
Therefore, $\Phi_\theta^{-1}(0) = \eh{\mathcal{A}}$.  \qed

\begin{theorem}
  \label{thm:2}
  With notation as above, the following statements are true:
  \begin{enumerate}
  \item \label{item:87} The set $\eh{\mathcal{B}}$ is a closed smooth
    submanifold of $\mathcal{B}$, $\stab{\mathcal{B}}$ is open in
    $\mathcal{B}$, $\stab{M}$ is open in $M$, the action of $\bar{K}$
    on $\eh{\mathcal{B}}$ is principal, and
    $\eh{p} : \eh{\mathcal{B}} \to \stab{M}$ is a smooth principal
    $\bar{K}$-bundle.
  \item \label{item:88} The action of $\bar{G}$ on
    $\stab{\mathcal{B}}$ is proper, $\stab{M}$ is a Hausdorff open
    subspace of $M$, the action of $\bar{G}$ on $\stab{\mathcal{B}}$
    is principal, and $\stab{p} : \stab{\mathcal{B}} \to \stab{M}$ is
    a holomorphic principal $\bar{G}$-bundle.
  \item \label{item:89} There exists a unique K\"ahler metric
    $\stab{h}$ on $\stab{M}$, such that
    $\eh{p}^*(\stab{\Theta}) = \eh{\Omega}$, where $\stab{\Theta}$ is
    the K\"ahler form of $\stab{h}$, $\eh{\Omega} = \eh{i}^*(\Omega)$,
    $\eh{i} : \eh{\mathcal{B}} \to \mathcal{B}$ is the inclusion map,
    and $\Omega$ is the K\"ahler form on $\mathcal{B}$.
  \end{enumerate}
\end{theorem}

\proof The stabiliser $G_\rho$ of any point $\rho$ of $\mathcal{B}$
equals $H$, and, by Theorem~\ref{thm:1}, the induced action of
$\bar{G}$ on $\mathcal{B}$ is principal.  Let
$\Psi_\theta : \mathcal{B} \to \Lie{K}^*$ be the restriction of
$\Phi_\theta$.  By Lemma~\ref{lem:18}, $\Phi_\theta$ is a moment map
for the action of $K$ on $\mathcal{A}$, hence $\Psi_\theta$ is a
moment map for the action of $K$ on $\mathcal{B}$.  Moreover,
$\Psi_\theta(\mathcal{B}) \subset \Phi_\theta(\mathcal{A}) \subset
\Ann{\Lie{H\cap K}}$, and
$\zm{\mathcal{B}} := \Psi_\theta^{-1}(0) = \Phi_\theta^{-1}(0) \cap
\mathcal{B} = \eh{\mathcal{A}} \cap \mathcal{B} = \eh{\mathcal{B}}$.
Finally, by Lemma~\ref{lem:16}, we have
$\zms{\mathcal{B}} := \zm{\mathcal{B}} G = \eh{\mathcal{B}}G =
\stab{\mathcal{B}}$, and
$P_G(\zm{\mathcal{B}},\zm{\mathcal{B}}) =
P_G(\eh{\mathcal{B}},\eh{\mathcal{B}}) \subset HK$.  Therefore, the
Theorem follows from Corollary~\ref{cor:4}.  \qed

\section{The line bundle on the moduli of stable representations}
\label{sec:line-bundle-moduli}

\subsection{Line bundles on quotients of vector spaces}
\label{sec:line-bundl-quot}

Let $V$ be a finite-dimensional $\C$-vector space,
$\pair{\cdot}{\cdot}$ a Hermitian inner product on $V$, and
$\Omega = -2 \Im(\pair{\cdot}{\cdot})$ its fundamental $2$-form.  We
will consider $V$ to be a K\"ahler manifold in the usual way.  Let $G$
be a complex Lie group, and $K$ a real Lie subgroup of $G$.  Suppose
that we are given a holomorphic linear right action of $G$, and that
the induced action of $K$ on $V$ preserves the Hermitian inner product
$\pair{\cdot}{\cdot}$ on $V$.

Let $\chi : G \to \units{\C}$ be a character of $G$, and suppose that
$\chi(K) \subset \UU{1}$.  Then, $\tang[e]{\chi}(\Lie{K})$ is
contained in the $\R$-subspace $\Lie{\UU{1}} = \sqrt{-1}\R$ of
$\Lie{\units{\C}} = \C$.  Fix a non-zero real number $\lambda$.  Let
$\alpha$ be the element of $\Lie{K}^*$ defined by
$\alpha(\xi) = -\frac{\sqrt{-1}}{\lambda} \tang[e]{\chi}(\xi)$ for all
$\xi \in \Lie{K}$.  Since $\chi(gag^{-1}) = \chi(a)$ for all
$a,g\in G$, $\alpha$ is $K$-invariant.  Therefore, by
Lemma~\ref{lem:14}, the map $\Phi_\alpha : V \to \Lie{K}^*$, which is
defined by
\begin{equation*}
  \label{eq:337}
  \Phi_\alpha(x)(\xi) =
  \frac{1}{2} \Omega(\xi^\sharp(x),x) + \alpha(\xi)
\end{equation*}
for all $x\in V$ and $\xi\in \Lie{K}$, is a moment map for the action
of $K$ on $V$.

Let $E$ denote the trivial holomorphic line bundle $V\times \C$ on
$V$.  Define a right action of $G$ on $E$ by setting
$(x,a)g = (xg,\chi(g)^{-1}a)$ for all $(x,a)\in E$ and $g\in G$.  Let
$\Gamma(E)$ denote the $\C$-vector space of smooth sections of $E$ on
$V$.  For each $\xi\in \Lie{G}$ and $s\in \Gamma(E)$, define another
section $\xi s\in \Gamma(E)$ by
\begin{equation*}
  \label{eq:366}
  (\xi s)(x) =
  \frac{\diff}{\diff t} \Bigr\vert_{t=0}
  \bigl(s(x\exp(t\xi))\exp(-t\xi)\bigr)
\end{equation*}
for all $x\in V$.  For every $x\in V$, define a Hermitian inner
product $h(x) : E(x) \times E(x) \to \C$ by
\begin{equation*}
  \label{eq:345}
  h(x)((x,a),(x,b)) = \exp(\lambda \norm{x}^2)a\bar{b}
\end{equation*}
for all $a,b\in \C$.  This gives a smooth Hermitian metric $h$ on $E$.

\begin{lemma}
  \label{lem:20}
  Let $\nabla$ be the canonical connection of the Hermitian
  holomorphic line bundle $(E,h)$ on $V$.  Then:
  \begin{enumerate}
  \item \label{item:95} For all $\xi\in \Lie{K}$ and $s\in \Gamma(E)$,
    we have
    $\nabla_{\xi^\sharp}(s) = \xi s - \lambda \sqrt{-1}
    \Phi_\alpha^\xi s$.
  \item \label{item:96} The first Chern form $c_1(E,h)$ of $\nabla$
    equals $- \frac{\lambda}{2\pi} \Omega$.
  \end{enumerate}
\end{lemma}

\proof Let $\xi\in \Lie{K}$.  Define a map $s_0 : V \to E$ by
$s_0(x) = (x,1)$ for all $x\in V$.  It is a holomorphic frame of $E$
on $V$.  We have
\begin{equation*}
  \label{eq:368}
  \nabla_v(s_0) = \lambda \pair{v}{x} s_0(x)
\end{equation*}
for all $x\in V$ and $v\in \tang[x]{V} = V$.  Therefore,
\begin{equation*}
  \label{eq:369}
  \nabla_{\xi^\sharp}(s_0)(x) =
  \lambda \pair{\xi^\sharp(x)}{x} s_0(x) =
  - \frac{\lambda\sqrt{-1}}{2} \Omega(\xi^\sharp(x),x) s_0(x).
\end{equation*}
On the other hand,
\begin{equation*}
  \label{eq:370}
  (\xi s_0)(x) = \tang[e]{\chi}(\xi) s_0(x) =
  \lambda \sqrt{-1} \alpha(\xi) s_0(x).
\end{equation*}
Thus,
\begin{equation*}
  \label{eq:371}
  (\xi s_0)(x) - \nabla_{\xi^\sharp}(s_0)(x) =
  \lambda \sqrt{-1} \bigl(
  \alpha(\xi) + \frac{1}{2} \Omega(\xi^\sharp(x),x)
  \bigr) s_0(x) =
  \lambda \sqrt{-1} \Phi_\alpha(x)(\xi) s_0(x).
\end{equation*}
It follows that
\begin{equation*}
  \label{eq:374}
  \xi s_0 - \nabla_{\xi^\sharp}(s_0) =
  \lambda \sqrt{-1} \Phi_\alpha^\xi s_0.
\end{equation*}
Now, let $s$ be an arbitrary element of $\Gamma(E)$.  Then, there
exists a smooth complex function $f$ on $V$, such that $s=fs_0$.  It
is easy to see that
\begin{equation*}
  \label{eq:372}
  \xi (fs_0) = \xi^\sharp(f)s_0 + f (\xi s_0).
\end{equation*}
Therefore,
\begin{align*}
  \xi s - \nabla_{\xi^\sharp}(s)
  &=
    \bigl( \xi^\sharp(f)s_0 + f (\xi s_0) \bigr) -
    \bigl( \xi^\sharp(f)s_0 + f \nabla_{\xi^\sharp}(s_0) \bigr) \\
  &=
    f (\xi s_0 - \nabla_{\xi^\sharp}(s_0)) =
    f \lambda \sqrt{-1} \Phi_\alpha^\xi s_0 =
    \lambda \sqrt{-1} \Phi_\alpha^\xi s.
\end{align*}
This proves (\ref{item:95}).

Let $\omega$ be the connection form of $\nabla$ with respect to the
holomorphic frame $s_0$ of $E$ on $V$, and $R$ the curvature form of
$\nabla$.  Then, $\omega = \lambda \partial N$, and
$R = \bar{\partial} \omega$, where $N : V \to \R$ is the smooth
function $x\mapsto \norm{x}^2$.  Therefore,
\begin{equation*}
  \label{eq:373}
  c_1(E,h) =
  \frac{\sqrt{-1}}{2\pi} R =
  - \frac{\lambda\sqrt{-1}}{2\pi} \partial \bar{\partial} N.
\end{equation*}
It is easy to see that $\sqrt{-1} \partial \bar{\partial} N = \Omega$.
Thus, $c_1(E,h) = - \frac{\lambda}{2\pi} \Omega$, as stated in
(\ref{item:96}). \qed

Let $H$ be a normal complex Lie subgroup of $G$, $\bar{G}$ the complex
Lie group $H\backslash G$, and $\pi: G \to \bar{G}$ the canonical
projection.  Let $\bar{K}$ be the compact subgroup $\pi(K)$ of
$\bar{G}$, and $\pi_K : K \to \bar{K}$ the homomorphism of real Lie
groups induced by $\pi$.

Let $X$ be a $G$-invariant open subset of $V$, $\zm{X}$ the closed
subset $\Phi_\alpha^{-1}(0)\cap X$ of $X$, and $\zms{X} = \zm{X}G$.
Denote by $Y$ the quotient topological space $X/G$, and let
$p : X \to Y$ be the canonical projection.  Let
$\zms{Y} = p(\zms{X})$, $\zms{p} : \zms{X} \to \zms{Y}$ the map
induced by $p$, and
$\zm{p} = \restrict{\zms{p}}{\zm{X}} : \zm{X} \to \zms{Y}$.

The subset $E_X = X\times \C$ is a $G$-invariant open subset of $E$.
Let $F$ denote the quotient topological space $E_X/G$, and
$q : E_X \to F$ the canonical projection.  There is a canonical
continuous surjection from $F$ to $Y$, and every fibre of this map has
a canonical structure of a $1$-dimensional $\C$-vector space.  Thus,
$F$ is a family of $1$-dimensional $\C$-vector spaces on $Y$.  Let
$\zm{F}$ (respectively, $\zms{F}$) denote the restriction of this
family to the subspace $\zm{Y}$ (respectively, $\zms{Y}$) of $Y$.  For
every $x\in X$, the map $q : E_X \to F$ restricts to a $\C$-isomorphism
$q(x) : E(x) \to F(p(x))$.

Note that if $H$ is contained in the kernel of the character
$\chi : G \to \units{\C}$, then we have an induced action of $\bar{G}$
on $E$, and hence on $E_X$.  If, moreover, the action of $\bar{G}$ on
$X$ is principal, then so is its action on $E_X$.  Thus, in that case,
there is a unique structure of a complex premanifold on $F$, such that
$q$ is a holomorphic submersion.  With this structure, the family $F$
of $1$-dimensional $\C$-vector spaces is a holomorphic line bundle on
$Y$(It is the holomorphic line bundle associated with the
holomorphic principal $\bar{G}$-bundle $p : X \to Y$, and the
character of $\bar{G}$ induced by $\chi : G \to \units{\C}$.).  For
every holomorphic (respectively, smooth) section $t$ of $F$ on any
open subset $V$ of $Y$, there exists a unique holomorphic
(respectively, smooth) section $s$ of $E_X$ on $p^{-1}(V)$, which is
$\bar{G}$-invariant (that is, $s(xa) = s(x)a$ for all $x\in p^{-1}(V)$
and $a\in \bar{G}$), such that $q(s(x)) = t(p(x))$ for all
$x\in p^{-1}(V)$.

\begin{proposition}
  \label{pro:30}
  Consider the context of Corollary~\ref{cor:4}.  Suppose that
  $G_x = H$ for all $x\in X$, the induced action of $\bar{G}$ on $X$
  is principal, $H\subset \Ker{\chi}$, and
  \begin{equation*}
    \label{eq:382}
    \Phi_\alpha(X) \subset \Ann{\Lie{H\cap K}}, \quad
    P_G(\zm{X},\zm{X}) \subset HK.
  \end{equation*}
  Then, there exists a unique smooth Hermitian metric $\zms{k}$ on the
  holomorphic line bundle $\zms{F}$ on $\zms{Y}$, such that
  $c_1(\zms{F},\zms{k}) = -\frac{\lambda}{2\pi} \zms{\Theta}$, where
  $\zms{\Theta}$ is the K\"ahler form on the open complex submanifold
  $\zms{Y}$ of $Y$.
\end{proposition}

\proof For every point $y\in \zms{Y}$, define
$\zms{k}(y) : F(y) \times F(y) \to \C$ by
$\zms{k}(y)(a,b) = h(x)(a',b')$, where $x$ is any point of
$\zm{p}^{-1}(y)$ and $a',b'\in E(x)$ are such that $q(a')=a$ and
$q(b')=b$.  Then, since $\zm{p} : \zm{X} \to \zms{Y}$ is a smooth
principal $\bar{K}$-bundle, and the metric $h$ is $K$-invariant, the
above rule gives a well-defined smooth Hermitian metric $\zms{k}$ on
$\zms{F}$.

Suppose $t$ is a smooth section of $\zms{F}$ on an open subset $V$ of
$\zms{Y}$, $y\in \zms{Y}$, and $w\in \tang[y]{Y}$.  We will define an
element $\nabla'_w(t)$ of $F(y)$ as follows.  Let
$x\in \zm{p}^{-1}(y)$, and choose $v\in \tang[x]{\zm{X}}$, such that
$\tang[x]{\zm{p}}(v) = w$.  Let $s$ be the unique $K$-invariant
section of $E$ on $\zm{p}^{-1}(V)$ which projects to $t$.  Define
$\nabla'_w(t) = q(\nabla_v(s))$.  If $x'\in \zm{p}^{-1}(y)$ and
$v'\in \tang[x']{\zm{X}}$ are two other choices, such that
$\tang[x']{\zm{p}}(v') = w$, then there exists a unique $g\in K$, such
that $x'=xg$.  Now, $v' - \tang[x]{\rho_g}(v)$ belongs to
$\Ker{\tang[x']{\zm{p}}}$, and is hence of the form $\xi^\sharp(x')$
for some $\xi\in \Lie{K}$.  Thus, by Lemma~\ref{lem:20},
\begin{equation*}
  \label{eq:383}
  \nabla_{v'}(s) =
  \nabla_{\xi^\sharp(x')}(s) +
  \nabla_{\tang[x]{\rho_g}(v)}(s) =
  \bigl(
  (\xi s)(x') - \lambda \sqrt{-1} \Phi_\alpha^\xi(x')s(x')
  \bigr) +
  \bigl(\nabla_v(s)\bigr)g,
\end{equation*}
since the action of $K$ preserves the metric $h$ on $E$, and hence its
canonical connection $\nabla$ also.  Now, since $s$ is $K$-invariant,
we have $\xi s = 0$, and since $x'\in \zm{X}$, we have
$\Phi_\alpha^\xi(x')=0$.  Therefore,
$\nabla_{v'}(s) = \bigl(\nabla_v(s)\bigr)g$, hence
$q(\nabla_{v'}(s)) = q(\bigl(\nabla_v(s)\bigr))$.  It follows that
$\nabla'_w(t)$ is well-defined.  Since $\zm{q}$ is a smooth principal
$\bar{K}$-bundle, this rule defines a smooth connection $\nabla'$ on
$\zms{F}$.

We claim that $\nabla'$ is the canonical connection of the Hermitian
holomorphic line bundle $(\zms{F},\zms{k})$ on $\zms{Y}$.  As $\nabla$
is compatible with the metric $h$ on $E$, and $K$ preserves $h$,
$\nabla'$ is compatible with the metric $\zms{k}$ on $\zms{F}$.
Therefore, we only need to check that $\nabla'$ is compatible with the
holomorphic structure on $\zms{F}$.  Let $t$ be a holomorphic section
of $\zms{F}$ on an open subset $V$ of $\zms{Y}$, $y\in V$, and
$w\in \tang[y]{Y}$.  We have to check that
$\nabla'_{\sqrt{-1}w}(t) = \sqrt{-1} \nabla'_{w}(t)$.  Let $s$ be the
$G$-invariant holomorphic section of $E$ on $\zms{p}^{-1}(V)$
corresponding to $t$.  Let $x\in \zm{p}^{-1}(y)$, and choose
$v\in \tang[x]{\zm{X}}$, such that $\tang[x]{\zm{p}}(v) = w$.  Then,
by Lemma~\ref{lem:15}, $\sqrt{-1}v = v' + \sqrt{-1}\xi^\sharp(x)$,
where $v' \in \tang[x]{\zm{X}}$ and $\xi\in \Lie{K}$.  By definition,
$\nabla'_w(t) = \nabla_v(s)$.  Similarly, since
$\tang[x]{\zm{p}}(v') = \tang[x]{p}(\sqrt{-1}(v-\xi^\sharp(x))) =
\sqrt{-1}w$, we have $\nabla'_{\sqrt{-1}w}(t) = \nabla_{v'}(s)$.  Now,
since $\nabla$ is compatible with the holomorphic structure on $E$, we
get
\begin{equation*}
  \label{eq:384}
  \nabla_{v'}(s) = \nabla_{\sqrt{-1}(v-\xi^\sharp(x))}(s) =
  \sqrt{-1} (\nabla_v(s)-\nabla_{\xi^\sharp(x)}(s))
\end{equation*}
But, as we saw above, $\nabla_{\xi^\sharp(x)}(s) = 0$.  It follows
that $\nabla'_{\sqrt{-1}w}(t) = \sqrt{-1} \nabla'_{w}(t)$.  This
proves the above claim.

Thus, the canonical connection $\nabla'$ on $(\zms{F},\zms{k})$ is the
descent of $\nabla$ through $\zm{p} : \zm{X} \to \zms{Y}$.  Therefore,
\begin{equation*}
  \label{eq:385}
  \zm{p}^*c_1(\zms{F},\zms{k}) =
  \zm{i}^* c_1(E,h),
\end{equation*}
where $\zm{i} : \zm{X} \to X$ is the inclusion.  But, by
Lemma~\ref{lem:20}, $c_1(E,h) = -\frac{\lambda}{2\pi} \Omega$, hence
\begin{equation*}
  \label{eq:386}
  \zm{p}^*c_1(\zms{F},\zms{k}) =
  -\frac{\lambda}{2\pi} \zm{i}^* c_1(E,h) =
  -\frac{\lambda}{2\pi} \zm{p}^* \zms{\Theta}.
\end{equation*}
As $\zm{p}$ is a smooth submersion, it follows that
$c_1(\zms{F},\zms{k}) = -\frac{\lambda}{2\pi} \zms{\Theta}$.  \qed

\subsection{The line bundle on the moduli space}
\label{sec:line-bundle-moduli-1}

We will follow the notation of
Section~\ref{sec:kahler-metric-moduli-1}.  Recall that $\theta$ is a
rational weight of $Q$.  Let $n$ be an integer $>0$, such that
$n(\theta_a - \mu_\theta(d)) \in \Z$ for all $a\in Q_0$.  Let
$\lambda = -n$.  Let $\chi : G \to \units{\C}$ be the character
\begin{equation*}
  \label{eq:387}
  \chi(g) = \prod_{a\in Q_0} \det(g_a)^{n(\mu_\theta(d)-\theta_a)}.
\end{equation*}
Then, $\chi(K) \subset \UU{1}$, and $H \subset \Ker{\chi}$, since
$\sum_{a\in Q_0} (\mu_\theta(d)-\theta_a)d_a = 0$.  Let
$\alpha = -\frac{\sqrt{-1}}{\lambda} \tang[e]{\chi}$.  Then,
\begin{equation*}
  \label{eq:388}
  \alpha(\xi) = \frac{\sqrt{-1}}{n} \tang[e]{\chi}(\xi) =
  \frac{\sqrt{-1}}{n}
  \sum_{a\in Q_0} n(\mu_\theta(d)-\theta_a) \tr{\xi_a} =
  \pair{\xi}{\eta},
\end{equation*}
where
$\eta = \bigl(\sqrt{-1}(\theta_a - \mu_\theta(d))\id{V_a}\bigr)_{a\in
  Q_0}$.  Thus,
\begin{equation*}
  \label{eq:389}
  \Phi_\alpha(\rho)(\xi) =
  \frac{1}{2} \Omega(\xi^\sharp(\rho),\rho) + \alpha(\xi) =
  \frac{1}{2} \Omega(\xi^\sharp(\rho),\rho) + \pair{\xi}{\eta} =
  \Phi_\theta(\rho)(\xi)
\end{equation*}
for all $\rho\in \mathcal{A}$ and $\xi\in \Lie{K}$.

Let $E$ be the trivial line bundle on $\mathcal{A}$ with the action of
$G$ defined by $\chi$ as above.  Let $\stab{F}$ be its quotient by $G$
on $\stab{M}$.  As above, $\stab{F}$ is a holomorphic line bundle on
$\stab{M}$.  Now, the following result is an immediate consequence of
Proposition~\ref{pro:30}.

\begin{theorem}
  \label{thm:3}
  Let $n$ be any positive integer, such that
  $n(\theta_a-\mu_\theta(d))\in \Z$ for all $a\in Q_0$.  There exists
  a unique smooth Hermitian metric $\stab{k}$ on the holomorphic line
  bundle $\stab{F}$ on $\stab{M}$, such that
  $c_1(\stab{F},\stab{k}) = \frac{n}{2\pi} \stab{\Theta}$, where
  $\stab{\Theta}$ is the K\"ahler form on $\stab{M}$.
\end{theorem}

\end{document}